\documentclass[a4paper]{amsart}
\usepackage{latexsym,bm,stmaryrd}
\usepackage{amsmath,amsthm,amsfonts,amssymb,mathrsfs,pb-diagram}
\usepackage[a4paper,hmargin=26mm,vmargin=28mm]{geometry}
\usepackage{microtype}
\usepackage{mathtools}
\usepackage{mathbbol,wasysym}
\usepackage[all,cmtip]{xy}

\usepackage[svgnames]{xcolor}
\usepackage{xparse}
\usepackage{listings}
\usepackage{tikz}
\usetikzlibrary{matrix}
\usepackage[misc]{ifsym}

\PassOptionsToPackage{unicode}{hyperref}
\PassOptionsToPackage{naturalnames}{hyperref}

\newcommand\blam{{\boldsymbol\lambda}}

\newcommand\bmu{{\boldsymbol\mu}}

\let\<=\langle
\let\>=\rangle
\def\({\big(}
\def\){\big)}

\def\Z{\mathbb{Z}}

\def\N{\mathbb{N}}

\def\lam{\lambda}
\def\Lam{\Lambda}
\def\Sym{\mathfrak{S}}
\def\eps{\varepsilon}

\newcommand\HH{\mathscr{H}}

\def\bn[#1,#2]{\begin{bmatrix}#1\\#2\end{bmatrix}}

\def\P{\mathscr{P}}

\def\we{\widetilde{e}}
\def\wf{\widetilde{f}}
\def\wi{\widehat{i}}
\def\wj{\widehat{j}}

\def\cp{\mathcal{P}}

\newcommand\whd{\widehat{\Delta}}
\newcommand\ua{\underline{a}}
\newcommand\ub{\underline{b}}

\DeclareMathOperator\lmod{\!-mod}
\DeclareMathOperator\Hom{Hom}

\DeclareMathOperator\soc{soc}

\DeclareMathOperator\head{hd}

\DeclareMathOperator\cha{char}

\DeclareMathOperator\pr{pr}

\DeclareMathOperator\Rep{Rep}

\DeclareMathOperator\wt{wt}

\DeclareMathOperator\res{res}
\DeclareMathOperator\ind{ind}
\DeclareMathOperator\charac{ch}
\DeclareMathOperator\infl{infl}

\title[Crystal of affine type $\widehat{A}_{\ell-1}$ and Hecke algebras]
{Crystal of affine type $\widehat{A}_{\ell-1}$ and Hecke algebras at a primitive $2\ell$th root of unity}
\author{Huang Lin}
  \address{School of Mathematical Sciences\\
  Zhejiang University\\
  Hangzhou, 310027, P.R. China}
  \email{3130100651@zju.edu.cn}
\author[Corresponding author]{Jun Hu\textsuperscript{\Letter}}\thanks{\Letter Jun Hu \qquad Email: junhu404@bit.edu.cn}\address{Key Laboratory of Mathematical Theory and Computation in Information Security, School of Mathematics and Statistics\\
  Beijing Institute of Technology\\
  Beijing, 100081, P.R. China}
\email{junhu404@bit.edu.cn}

\numberwithin{equation}{section}
\newtheorem{prop}[equation]{Proposition}
\newtheorem{thm}[equation]{Theorem}
\newtheorem{cor}[equation]{Corollary}

\newtheorem{lem}[equation]{Lemma}
\newtheorem{exmp}[equation]{Example}

\theoremstyle{definition}
\newtheorem{dfn}[equation]{Definition}
\theoremstyle{remark}

\begin{document}

\begin{abstract}
Let $\ell\in\N$ with $\ell>1$. In this paper we give a new realization of the crystal of affine type $\widehat{A}_{\ell-1}$ using the modular representation theory of the affine Hecke algebras $\HH_n$ of type $A$ and their level two cyclotomic quotients with Hecke parameter being a primitive $2\ell$th root of unity. We construct ``hat'' versions of $i$-induction and $i$-restriction functors on the category $\Rep_I(\HH_n)$ of finite dimensional integral modules over $\HH_n$, which induce Kashiwara operators on a suitable subgroup of the Grothendieck groups of $\Rep_I(\HH_n)$. For any simple module $M\in\Rep_I(\HH_n)$, we prove that the simple submodules of $\res_{\HH_{n-2}}^{\HH_n}M$ which belong to $\widehat{B}(\infty)$ (Definition \ref{infty}) occur with multiplicity two. The main results generalize the earlier work of Grojnowski and Vazirani on the relations between the crystal of $\widehat{\mathfrak{sl}}_{\ell}$ and the affine Hecke algebras of type $A$ at a primitive $\ell$th root of unity.
\end{abstract}

\maketitle
\setcounter{tocdepth}{1}
\tableofcontents

\section{Introduction}

Let $1<\ell\in\N$. Let $B(\hat{\Lam}_0)$ be the crystal of the integral highest weight module $L(\hat{\Lam}_0)$ of the affine Lie algebra $\widehat{\mathfrak{sl}}_{\ell}$ and ${B}({\Lam}_0+{\Lam}_\ell)$ the crystal of the integral highest weight module $L(\Lam_0+\Lam_\ell)$ of the affine Lie algebra $\widehat{\mathfrak{sl}}_{2\ell}$, where $\hat{\Lam}_0$ and $\Lam_0,\Lam_\ell$ are the fundamental dominant weights of $\widehat{\mathfrak{sl}}_{\ell}$ and $\widehat{\mathfrak{sl}}_{2\ell}$ respectively. By \cite[(4.2)]{NS}, there is a natural embedding $\iota: B(\hat{\Lam}_0)\cup\{0\}\hookrightarrow{B}({\Lam}_0+{\Lam}_\ell)\cup\{0\}$ which is defined by \begin{equation}\label{iota}
\iota(\widetilde{f}_{i_n+\ell\Z}\cdots \widetilde{f}_{i_1+\ell\Z}1_{\hat{\Lam}_0})=\widetilde{f}_{i_n+2\ell\Z}\widetilde{f}_{i_n+\ell+2\ell\Z}\cdots \widetilde{f}_{i_1+2\ell\Z}\widetilde{f}_{i_1+\ell+2\ell\Z}1_{{\Lam}_0+{\Lam}_\ell},\,\,\forall\,n\in\N, i_1,\cdots,i_n\in\Z ,
\end{equation}
such that $\iota(B(\hat{\Lam}_0))\subseteq {B}({\Lam}_0+{\Lam}_\ell)$.

The above embedding $\iota$ has some important combinatorial and representation theorietic implication. Recall that $B(\hat{\Lam}_0)$ has a realization in terms of the set $\mathcal{K}_{0}$ of $\ell$-restricted partitions (or equivalently, Kleshchev partitions), while ${B}({\Lam}_0+{\Lam}_\ell)$ has a realization in terms of the set ${\mathcal{K}}_{0,\ell}$ of Kleshchev bipartitions with respect to $(\sqrt[2\ell]{1};1,-1)$ (\cite[Definition 2.3]{Ariki:sim}, \cite[Page 605, Definition]{AM}), where $\sqrt[2\ell]{1}$ denotes a primitive $2\ell$th root of unity. Thus for each $n\in\N$, $\iota$ defines an injection (\cite[Corollary 6.9]{Hu5}) from the subset $\mathcal{K}_{0}(n)$ into the subset ${\mathcal{K}}_{0,\ell}(2n)$, such that if $\emptyset\overset{i_1}{\rightarrow}\cdot\overset{i_2}{\rightarrow}\cdot\cdots\overset{i_n}{\rightarrow}\lam$ is a path in Kleshchev's good lattice of $\mathcal{K}_{0}$ then  $(\emptyset,\emptyset)\overset{i_1}{\rightarrow}\cdot\overset{i_1+\ell}{\rightarrow}\cdot\overset{i_2}{\rightarrow}\cdot\overset{i_2+\ell}{\rightarrow}\cdot\cdots
\overset{i_n}{\rightarrow}\cdot\overset{i_n+\ell}{\rightarrow}\iota(\lam)$ is a path in Kleshchev's good lattice of ${\mathcal{K}}_{0,\ell}$, where $\mathcal{K}_{0}(n)$ (resp., ${\mathcal{K}}_{0,\ell}(2n)$) denotes the set of $\ell$-restricted partitions of $n$ in $\mathcal{K}_0$ (resp., of Kleshchev bipartitions of $2n$ in $\mathcal{K}_{0,\ell}$). Furthermore, $\iota(B(\hat{\Lam}_0))$ coincides with the fixed point subset of  ${B}({\Lam}_0+{\Lam}_\ell)$ under the automorphism ``${\rm{h}}$'' induced by the Dynkin diagram automorphism $i\mapsto i+\ell+2\ell\Z$ for all $i\in\Z/2\ell\Z$. The subset $\mathcal{K}_0(n)$ gives a labelling of simple modules for the Iwahori--Hecke algebra
of type $A_{n-1}$ (i.e., associated to the symmetric group $\Sym_n$) at a primitive $\ell$th root of unity $\sqrt[\ell]{1}$, while the subset ${\mathcal{K}}_{0,\ell}(2n)$ gives a labelling of simple modules for the Iwahori--Hecke algebra of type $B_{2n}$ at a primitive $2\ell$th root of unity $\sqrt[2\ell]{1}$. This gives a first clue on the connection between the modular representations of the Iwahori--Hecke algebras of type $A$ and of type $B$ at different roots of unity via (\ref{iota}).

The second implication of (\ref{iota}) involves the Iwahori--Hecke algebras of type $D$ at root of unity. Let $F$ be an algebraically closed field with $\cha F\neq 2$ and $1\neq q\in F^\times$. Recall that the Iwahori--Hecke algebra $\HH_q(B_n)$ of type $B_n$ is the unital associative $F$-algebra generated by $T_0,T_1,\cdots,T_{n-1}$ which satisfy the following relations: $$\begin{aligned}
& T_0^2=1,\,\,(T_r-q)(T_r+1)=0,\,\,\forall\,1\leq r<n,\\
& T_i T_{i+1} T_i=T_{i+1}T_iT_{i+1},\,\,\forall\,1\leq i<n-1,\\
& T_iT_j=T_jT_i,\,\,\forall\, 1\leq i<j-1<n-1,\\
& T_0T_1T_0T_1=T_1T_0T_1T_0 .
\end{aligned}
$$
The $F$-subalgebra of $\HH_q(B_n)$ generated by $T_0T_1T_0, T_1,\cdots,T_{n-1}$ is isomorphic to the Iwahori--Hecke algebra $\HH_q(D_n)$ associated to the Weyl group of type $D_n$. In a series of earlier works \cite{Hu1}, \cite{Hu2}, \cite{Hu3}, \cite{Hu4}, the second author has initiated the study of the modular representations of $\HH_q(D_n)$ using the Clifford theory between $\HH_q(B_n)$ and $\HH_q(D_n)$, with the aim of computing the decomposition numbers of $\HH_q(D_n)$ in terms of the
decomposition numbers of $\HH_q(B_n)$. Let $\P(n)$ and $\cp(n)$ be the set of bipartitions and partitions of $n$ respectively.
Let $\{S^{\blam}|\blam\in\P(n)\}$ and $\{D^{\blam}\neq 0|\blam\in\P(n)\}$ be the set of Specht modules and simple modules of $\HH_q(B_n)$ respectively, where
$D^\blam$ is defined in \cite{DJM1} as certain quotient of $S^\blam$. In the semisimple case, $S^{\blam}\downarrow_{\HH_q(D_n)}$ splits into a direct sum of two simple submodules $S_+^\blam\oplus S_-^\blam$ whenever $\blam=(\lam,\lam)$ for some $\lam\in \cp(n/2)$.
Set $\mathscr{K}(n):=\{\blam\in\P(n)|D^\blam\neq 0\}$.
By \cite{Hu1}, for each $\blam\in{\mathscr{K}}(n)$, $D^{\blam}\downarrow_{\HH_q(D_n)}$ either remains simple, or splits into a direct sum of two simple submodules $D_+^\blam\oplus D_{-}^\blam$. The most interesting and not well-understood case is when $n$ is even and the Hecke parameter $q$ is a primitive $2\ell$-th root of unity. In that case, $\mathscr{K}(n)=\mathcal{K}_{0,\ell}(n)$, $D^{\blam}\downarrow_{\HH_q(D_n)}$ splits if and only if $\blam={\rm h}(\blam)$, and if and only if $\blam=\iota(\mu)$ for some $\mu\in\mathcal{K}_0(n/2)$.
Moreover, the set $$
\bigl\{D^{\blam}\downarrow_{\HH_q(D_n)}\bigm|\blam\in{\mathcal{K}}_{0,\ell}(n)/\!\!\sim, {\rm h}(\blam)\neq\blam\bigr\}\sqcup \bigl\{D_+^{\iota(\mu)},D_-^{\iota(\mu)}\bigm|\mu\in\mathcal{K}_0(n/2)\bigr\}
$$
is a complete set of pairwise non-isomorphic simple $\HH_q(D_n)$-modules, where $\blam\sim\bmu$ if and only if $\bmu={\rm h}(\blam)$.

A major challenging problem in the understanding of the modular representations of $\HH_q(D_n)$ when $n$ is even and $q$ is a primitive $2\ell$th root of unity is to understand the decomposition numbers $[S^{(\lam,\lam)}_+:D_+^{\iota(\mu)}]$, where $\lam$ is a partition of $n/2$ and $\mu$ is an $\ell$-restricted partition of $n/2$. We suspect that (\ref{iota}) reveals not only the bijection between $\mathcal{K}_0(n/2)$ and the set $\{\blam\in{\mathcal{K}}_{0,\ell}(n)|\text{$D^{\blam}\downarrow_{\HH_q(D_n)}$ splits}\}$ (a fact which was first obtained in \cite{Gec}), but also indicates some possible connections between the following three (type $A$, type $B$ and type $D$) decomposition numbers $$
[S^{\lam}:D^{\mu}],\quad [S^{(\lam,\lam)}:D^{\iota(\mu)}],\quad [S^{(\lam,\lam)}_+:D_+^{\iota(\mu)}] ,$$
where $S^\lam, D^\mu$ denote the Specht module labelled by $\lam$ and the simple module labelled by $\mu$ of $\HH_{\sqrt[\ell]{1}}(\Sym_{n/2})$ respectively.

By the celebrated work of Ariki, Lascoux, Leclerc and Thibon (\cite{Ariki:can}, \cite{LLT}), the decomposition numbers for the Hecke algebras of type $A$, type $B$ or more generally of type $G(r,1,n)$ when $\cha F=0$ can be computed through the calculation of the canonical bases of certain integral highest weight modules over $\widehat{\mathfrak{sl}}_{\ell}$, where $\ell$ is the multiplicative order of the Hecke parameter $q$. More recently, using the theory of quiver Hecke algebras, Brundan and Kleshchev \cite{BK:GradedKL} show that Ariki-Lascoux-Leclerc-Thibon's theory can be upgraded into a $\Z$-graded setting. In \cite{G} and \cite{GV}, Grojnowski and Vazirani give a new approach to the modular representations of affine Hecke algebras and their cyclotomic quotients over field of any (possibly positive) characteristic. In their approach a new realization of the crystals of affine type $\widehat{A}_{\ell-1}$ is obtained using the modular representation theory of affine Hecke algebras and their cyclotomic quotients at a primitive $\ell$th root of unity, where the Kashiwara operators $\we_i, \wf_i$ for the crystal are realized as the functors of taking socle of $i$-restriction and of taking head of $i$-induction. Grojnowski and Vazirani's approach also successfully applied to other Hecke algebras such as Hecke-Cifford algebras (\cite{BK1}, \cite{Tsu}) and KLR algebras (\cite{LV}). Note that the Dynkin diagram $A_{2\ell}^{(2)}$ (resp., $D_{\ell+1}^{(2)}$) can be obtained by $A_{2\ell}^{(1)}$ (resp., $A_{2\ell-1}^{(1)}$) by some diagram automorphisms (\cite{Hu06}).
In this paper, motivated by the embedding (\ref{iota}) (which can be viewed as some sort of diagram folding), we give a new realization of the crystal of affine type $\widehat{A}_{\ell-1}$ using the modular representation theory of the affine Hecke algebras of type $A$ and their level two cyclotomic quotients (i.e., Iwahori--Hecke algebra of type $B$) at a primitive $2\ell$th root of unity.
We realize the Kashiwara operators for the crystal as the functors of taking socle of certain two-steps restriction and of taking head of certain two-steps induction. For any simple module $M\in\HH_n\lmod$, we prove that the simple submodules of $\res_{\HH_{n-2}}^{\HH_n}M$ which belong to $\widehat{B}(\infty)$ (Definition \ref{infty}) occur with multiplicity two. The main tool we use is the theory of real simple modules developed by Kang, Kashiwara, et al. in recent papers (\cite{KKK}, \cite{KKKO}). The theory is originally built for the quiver Hecke algebras, but can be transformed into the setting of the affine Hecke algebras of type $A$ using the categorical equivalence between finite dimensional representations of affine Hecke algebras of type $A$ and of the quiver Hecke algebras of type $A$ established in \cite{HL}.

\medskip

The content of the paper is organized as follows. In Section 2, we first recall some basic knowledge (like intertwining elements, convolution product, the functor $\Delta_{\underline{b}}$ etc.) about the non-degenerate affine Hecke algebra of type $A$. Then we introduce in Definitions \ref{2Steps} and \ref{GKato} the notion of two-steps induction functors $e_{\wi}$, the generalized Kato modules $L(\wi^m)$ and the function $\eps_{\wi}$. We give in Lemmas \ref{irre1}, \ref{DeltaCharacter}, \ref{epswi01} and Corollary \ref{epsHat2} a number of results on the functor $\whd_{\wi^m}$ and the function $\eps_{\wi}$. In Section 3 we first recall the notion of real simple module and some of their main properties established in \cite{KKKO}. Then we give Lemmas \ref{real1}, \ref{ijcommutes}, Corollary \ref{realCor} which provide a number of new real simple modules and their nice properties. In particular, our Corollary \ref{headCor} fixes a gap of \cite[6.3.2]{Klesh:book} in the case when $p>2$, see \cite{Klesh:book2}. In Section 4 we study the two-steps version $\we_{\wi}, \wf_{\wi}$ of the functors $\we_i,\wf_i$ of taking socle of $i$-restriction and taking head of $i$-induction. The main result in this section is Proposition \ref{RedBao}, where we prove that the socle of $e_ie_{i+\ell}M$ is isomorphic to $\we_i\we_{i+\ell}M$ for any finite dimensional simple $\HH_n$-module $M$. In addition, we also give generalization in our ``hat" setting of some results in \cite{Klesh:book} for the functors $\we_i, \wf_i, \Delta_{i^a}$ and the function $\eps_i$. In Section 5 we give the first two main results Theorems \ref{mainthm1a}, \ref{mainthm1b} of this paper, which give a new realization of the crystal of $U_v(\widehat{\mathfrak{sl}}_{\ell})^{-}$ using the modular representation theory of the affine Hecke algebras $\HH_n$ of type $A$ at a primitive $2\ell$th root of unity. In Section 6 we give the third main result Theorem \ref{mainthm3} of this paper, which give a new realization of the crystal of the integral highest weight module $V(\hat{\Lam}_0)$ of $U_v(\widehat{\mathfrak{sl}}_{\ell})$ using the modular representation theory of the Iwahori--Hecke algebra of type $B$ at a primitive $2\ell$th root of unity. We also obtain several multiplicity two results in Theorems \ref{mainthm2a}, \ref{mainthm2b} and Corollary \ref{maincor2} about certain two-steps restrictions and inductions.

\bigskip

\centerline{Acknowledgements}
\bigskip

The research was supported by the National Natural Science Foundation of China (No. 12171029).
\bigskip

\section{Preliminary}

Throughout this paper, let $F$ be an algebraically closed field and $1\neq q\in F^{\times}$. Let $\HH_n:=\HH_n(q)$ be the {\bf non-degenerate type $A$ affine Hecke algebra} over $F$ with Hecke parameter $q$. By definition, $\HH_n$ is the unital associative $F$-algebra with generators $T_1,\dots,T_{n-1}$, $X_1^{\pm 1},\dots,X_n^{\pm 1}$ and relations:
\begin{align}
(T_i-q)(T_i+1)&=0,\,\,\quad 1\leq i<n, \label{1} \\
 T_iT_{i+1}T_i&=T_{i+1}T_iT_{i+1},\,\,\quad 1\leq i\leq n-2,\label{2}\\
  T_iT_k&=T_kT_i,\,\,\quad |i-k|>1, \label{3}\\
 X_i^{\pm 1}X_k^{\pm 1}&=X_k^{\pm 1}X_i^{\pm 1},\,\,\quad 1\leq i,k\leq n, \label{4}\\
 X_kX_k^{-1}&=1=X_k^{-1}X_k,\,\, \quad 1\leq k\leq n, \label{5}\\
 T_iX_k&=X_kT_i,\,\,\quad k\neq i,i+1, \label{6}\\
 X_{i+1}&=q^{-1}T_iX_{i}T_i,\,\,\quad 1\leq i<n . \label{7}
\end{align}
Note that one can also replace the last relation above with the following: \begin{align}
X_{i+1}T_i=T_iX_i+(q-1)X_{i+1},\,\,\quad 1\leq i<n . \label{7a}
\end{align}

Let $\ast$ be the anti-isomorphism of $\HH_n$ which is defined on generators by $T_i^*=T_i$, $X_j^*=X_j$ for any $1\leq i<n, 1\leq j\leq n$. For any $a,b\in\N$ with $1\leq a<b$, we denote by $\HH_{\{a+1,a+2,\cdots,b\}}$ the affine Hecke algebra which is isomorphic to $\HH_{b-a}$ and whose defining generators and relations are obtained from that of $\HH_{b-a}$ by shifting all the subscript upwards by $a$. In particular, $\HH_n=\HH_{\{1,2,\cdots,n\}}$. For later use, we need certain elements of $\HH_n$ which are called intertwining elements.

\begin{dfn}\text{(\cite[(4.10)]{BK:GradedKL}, \cite{Lu}, \cite{Rog})} For each $1\leq k<n$, we define the $k$th intertwining element $\Phi_k$ to be: $$
\Phi_k:=(1-X_kX_{k+1}^{-1})T_k+1-q=T_k(1-X_{k+1}X_k^{-1})+(q-1)X_{k+1}X_k^{-1} .
$$
\end{dfn}
Note that the element $\Phi_k$ defined above is the same as $\Theta_k^*$ in the notation of \cite[(4.10)]{BK:GradedKL}. These elements have the following nice properties which are mostly easy to check.

\begin{lem}\text{(\cite[Proposition 5.2]{Lu}, \cite[(4.11),(4.12),(4.13)]{BK:GradedKL})}\label{Intertwin} 1) For any $1\leq k<n$, $$\Phi_k^2=(q-X_{k+1}X_k^{-1})(q-X_kX_{k+1}^{-1}).$$

2) For any $1\leq k<n-1$, $\Phi_k\Phi_{k+1}\Phi_k=\Phi_{k+1}\Phi_k\Phi_{k+1}$, and $\Phi_k\Phi_j=\Phi_j\Phi_k$ for any $1\leq j<n$ with $|j-k|>1$;

3) For any $1\leq k,j<n$ with $j\neq k,k+1$, we have that $$
\Phi_kX_k=X_{k+1}\Phi_k,\quad \Phi_kX_{k+1}=X_{k}\Phi_k,\quad \Phi_kX_j=X_{j}\Phi_k .
$$

4) For any $1\leq k,j<n$ with $j\neq k-1,k,k+1$, we have $\Phi_k T_j=T_j\Phi_k$.
\end{lem}

Let $\Sym_n$ be the symmetric group on $\{1,2,\cdots,n\}$. For each $1\leq k<n$, set $s_k:=(k,k+1)$. For any $w\in\Sym_n$ and any reduced expression $s_{i_1}\cdots s_{i_m}$ of $w$, we define $\Phi_w:=\Phi_{i_1}\cdots\Phi_{i_m}$. Then $\Phi_w$ depends only on $w$ but not on the choice of the reduced expression of $w$ because of the braid relations 2) in Lemma \ref{Intertwin}.

\begin{cor} Let $w\in\Sym_n$. If $1\leq k\leq n$ then $\Phi_w X_k=X_{w(k)}\Phi_w$.
\end{cor}

\begin{proof} This follows from Lemma \ref{Intertwin} 3).
\end{proof}

\begin{cor}\label{PhiTcommute} Let $w\in\Sym_n$ and $1\leq k<n$. Suppose $w(k+1)=w(k)+1$. Then $\Phi_w T_k=T_{w(k)}\Phi_w$.
\end{cor}

\begin{proof} It is clear that $ws_k=s_{w(k)}w$. Since $w(k)<w(k+1)$, we have $\ell(ws_k)=\ell(w)+1=\ell(s_{w(k)}w)$. Therefore,
$\Phi_w\Phi_k=\Phi_{w(k)}\Phi_w$ by Lemma \ref{Intertwin}.

By definition, we have $$\begin{aligned}
\Phi_w\Phi_k&=\Phi_w\bigl(T_k(1-X_{k+1}X_k^{-1})+(q-1)X_{k+1}X_k^{-1}\bigr),\\
\Phi_{w(k)}\Phi_w&=\Bigl(T_{w(k)}(1-X_{w(k+1)}X_{w(k)}^{-1})+(q-1)X_{w(k+1)}X_{w(k)}^{-1}\Bigr)\Phi_w\\
&=T_{w(k)}\Phi_w(1-X_{k+1}X_k^{-1})+(q-1)\Phi_w X_{k+1}X_k^{-1}.
\end{aligned}$$
Now $\Phi_w\Phi_k=\Phi_{w(k)}\Phi_w$ implies that $$
\Phi_wT_k(1-X_{k+1}X_k^{-1})=T_{w(k)}\Phi_w(1-X_{k+1}X_k^{-1}) .
$$
Since $1-X_{k+1}X_k^{-1}$ is not a zero divisor of $\HH_n$, it follows that $\Phi_wT_k=T_{w(k)}\Phi_w$.
\end{proof}

Let $\HH_n\lmod$ be the category of finite dimensional $\HH_n$-modules. For any $M\in\HH_n\lmod$, we denote by $\head(M)$ the head of $M$ (i.e., the maximal semisimple quotient of $M$), and by $\soc M$ the socle of $M$ (i.e., the maximal semisimple submodule of $M$).

\begin{dfn} Let $m,n\in\N$. For each $M\in\HH_m\lmod$, $N\in\HH_n\lmod$, we define the convolution product $M\circ N$ of $M$ and $N$ to be: $$
M\circ N:=\ind_{m,n}^{m+n}M\boxtimes N \in\HH_{m+n}\lmod .
$$
Set $M\triangledown N:=\head(M\circ N)$. For any $k\in\N$, we define $N^{\circ k}:=\underbrace{N\circ N\circ\cdots\circ N}_{\text{$k$ copies}}$.
\end{dfn}

Let $m,n,k\in\N$. It is well-known that for any $M\in\HH_m\lmod$, $N\in\HH_n\lmod$ and $K\in\HH_k\lmod$, there is a canonical $\HH_{m+n+k}$-module isomorphism: \begin{equation}\label{transitive}
(M\circ N)\circ K\cong M\circ (N\circ K) .
\end{equation}

\begin{dfn}\label{sigmatau} Let $\sigma$ and $\tau$ be the automorphism of $\HH_n$ which is defined on generators as follows:
\begin{align*}
    &\sigma:\quad\,\, T_i\mapsto -q T_{n-i}^{-1},\quad X_j\mapsto X_{n+1-j}, \\
    &\tau :\quad\,\, T_i\mapsto T_i,\quad \quad X_j\mapsto -X_{j},
\end{align*}
for all $i=1,\cdots,n-1, j=1,\cdots,n$.
\end{dfn}

For any composition $\nu=(\nu_1,\nu_2,\cdots,\nu_r)$ of $n$, we define $\nu^*:=(\nu_r,\cdots,\nu_2,\nu_1)$ and $$
\HH_{\nu}:=\HH_{\nu_1}\boxtimes\HH_{\{\nu_1+1,\cdots,\nu_1+\nu_2\}}\boxtimes\cdots\boxtimes\HH_{\{n-\nu_r+1,\cdots,n\}} ,
$$
which is parabolic subalgebra of $\HH_n$. If $M\in\HH_{\nu}\lmod$, then we can twist the action with $\sigma$ to get a new module $M^\sigma\in\HH_{\nu^*}\lmod$.

\begin{lem} Let $M\in\HH_m\lmod$ and $N\in\HH_n\lmod$. Then $$
(M\circ N)^{\sigma}\cong N^{\sigma}\circ M^{\sigma} .
$$
\end{lem}
\medskip

Let $1<e\in\N$, $I:=\Z/e\Z$, and $q$ is a primitive $e$th root of unity in $F$.
When such a $q$ is fixed, we can identified $I$ with $q^I\subset F^{\times}$ via $i\mapsto q^i$.\medskip

\begin{dfn}\label{def:intrep} Let $\Rep_I\HH_n$ be the full subcategory of $\HH_n\lmod$ consisting of all modules $M$ such that all eigenvalues of $X_1,\cdots,X_n$ on $M$ belong to $q^I$. If $M\in \Rep_I\HH_n$ then we say that the $\HH_n$-module $M$ is integral.
\end{dfn}

In \cite{G} and \cite{GV}, Grojnowski and Vazirani work for both the non-degenerate type $A$ affine Hecke algebras and the degenerate type $A$ affine Hecke algebras. The theory for the non-degenerate case is parallel to the theory for the degenerate case. Kleshchev \cite{Klesh:book} gives an excellent account and explanation of Grojnowski's approach in the case of degenerate type $A$ affine Hecke algebras. In most of the time the results and their proof in Kleshchev's book \cite{Klesh:book} can be transformed into the case of non-degenerate affine Hecke algebras without any difficulty. In such case we shall simply cite them as ``the non-degenerated version" of the corresponding result in \cite{Klesh:book} whenever we can not find a suitable reference elsewhere.

Let $M\in\Rep_I\HH_n$. For any $\ua:=(a_1,\cdots,a_n)\in I^n$, let  $$
M_{\ua}:=\bigl\{x\in M\bigm|\text{$(X_j-q^{a_j})^Nx=0$ for any $1\leq j\leq n$ and $N\gg 0$}\bigr\} .
$$
We define the character of $M$ to be \begin{equation}
\charac M:=\sum_{\underline{a}\in I^n}\dim M_{\underline{a}}[a_1,a_2,\cdots,a_n] .
\end{equation}

For any  $1\leq t\leq n$ and $\ub:=(b_t,b_{t+1},\cdots,b_n)\in I^{n-t+1}$, we define \begin{equation}\label{Delta1}
\Delta_{\ub}M:=\oplus_{\substack{\ua=(a_1,\cdots,a_n)\in I^n\\ a_j=b_j, \forall\,t\leq j\leq n}}M_{\ua} .
\end{equation}
If $t=n$, $b_n=b$, then we write $\Delta_bM$ instead of $\Delta_{\ub}M$, see \cite[(5.1)]{Klesh:book}. In this case, $\Delta_bM$ is simply the generalized $q^b$-eigenspace of $X_n$ on $M$. That is, \begin{equation}\label{Delta2}
\Delta_bM=\oplus_{\substack{\ua=(a_1,\cdots,a_n)\in I^n\\ a_n=b}}M_{\ua}=\bigl\{x\in M\bigm|\text{$(X_n-q^{b})^Nx=0$ for any $N\gg 0$}\bigr\} .
\end{equation}

It is obvious that $\Delta_{b}M$ is $\HH_{n-1,1}$-stable. In general, note that though $\Delta_{\ub}M$ is still $\HH_{t-1}$-stable, it is not clear whether it is $\HH_{t-1,n-t+1}$-stable or not unless $b_t = \dots = b_{n-1} = b_n = b$. So in general $\Delta_{\ub}$ does not define a functor from $\HH_n\lmod$ to $\HH_{t-1,n-t+1}\lmod$.

Let $1\leq n\in\N$. Following \cite[\S8]{G} (see also \cite[(5.6)]{G}), we define the functor $$
e_i:=\res^{n-1,1}_{n-1}\circ\Delta_i:\,\, \Rep_I\HH_n\rightarrow \Rep_I\HH_{n-1} , $$
The functor $e_i$ is denoted by $e_i^*$ in the notation of \cite[\S8]{G}. Note that if $e_iM\neq 0$, then $e_iM$ is a self-dual module by the non-degenerate version of \cite[Lemma 7.3.1, Remark 8.2.4, Theorem 8.2.5]{Klesh:book}.

Following \cite{G} and \cite{Klesh:book}, for any $i\in I$ and any simple module $M\in\Rep_I\HH_n$, we define $$
\widetilde{e}_iM:=\soc e_iM\in\Rep_I\HH_{n-1},\quad \widetilde{f}_iM:=\head(M\circ L(i))\in\Rep_I\HH_{n+1},
$$
where $L(i)$ is the simple $\HH_1 = F[x_1]$-module on which $x_1$ acts as multiplication by $q^i$.
By \cite[Theorem 9.4]{G} (or rather the non-degenerate version of \cite[Lemma 5.15, Corollary 5.17]{Klesh:book}), we know that \begin{equation}\label{widetildeef}
\text{$\widetilde{f}_iM$ is always nonzero and simple, while $\widetilde{e}_iM$ is either $0$ or a simple module}.
\end{equation}
let $\mathbf{1}$ denote the trivial simple module of $\HH_0\cong K$. For any $(i_1,\cdots,i_n)\in I^n$, we define \begin{equation}\label{simples}
L(i_1,i_2,\cdots,i_n):=\widetilde{f}_{i_n}\cdots \widetilde{f}_{i_2}\widetilde{f}_{i_1}\mathbf{1}.
\end{equation}
Then $0\neq L(i_1,\cdots,i_n)\in\Rep_I\HH_n$ is a simple module.

\begin{lem}\label{2functorsIso} Let $i,j\in I$ with $i-j\neq\pm1$. There is an isomorphism of functors: $e_ie_{j}\cong e_{j}e_i$.
\end{lem}

\begin{proof} Since the case $i = j$ is trivial, we assume that $i \neq j$. For any $M\in\Rep_I\HH_n$, we use $\Phi: M\rightarrow M$ to denote the map given by left multiplication with $\Phi_{n-1}$. By Lemma \ref{Intertwin}, it is clear that $\Phi$ is an $\HH_{n-2}$-module homomorphism. We claim that $\Phi(e_ie_{j}M)=e_{j}e_iM$.

In fact, for any $x\in e_ie_{j}M$, $(X_{n-1}-q^i)^kx=0=(X_n-q^{j})^kx$ for $k\gg 0$. Therefore, using Lemma \ref{Intertwin}, for any $k\gg 0$, $$\begin{aligned}
& (X_{n-1}-q^{j})^k\Phi(x)=(X_{n-1}-q^{j})^k\Phi_{n-1}(x)=\Phi_{n-1}(X_{n}-q^{j})^k(x)=0,\\
& (X_{n}-q^{i})^k\Phi(x)=(X_{n}-q^{i})^k\Phi_{n-1}(x)=\Phi_{n-1}(X_{n-1}-q^{i})^k(x)=0,\\
\end{aligned}$$
which implies that $\Phi(e_ie_{j}M)\subseteq e_{j}e_iM$. Similarly, $\Phi(e_{j}e_iM)\subseteq e_ie_{j}M$. To finish the proof, it suffices to show that left multiplication with $\Phi_{n-1}^2$ defines an $\HH_{n-2}$-module automorphism of $e_ie_{j}M$. But this follows from Lemma \ref{Intertwin} 1) and the assumption that $i-j\neq\pm1$ and $q=\sqrt[e]{1}$ for some $e>1$.
\end{proof}

Following \cite{G} and \cite{Klesh:book}, for any $i\in I$ and any simple module $M\in\HH_n\lmod$, we define $$
\eps_i(M):=\max\{m\geq 0|\Delta_{i^m}M\neq 0\}=\max\{m\geq 0|\widetilde{e}_i^mM\neq 0\}.
$$

\begin{lem}\label{eijcommutes} Let $i,j\in I$ with $i-j\neq 0,\pm1$. For any simple module $M\in\HH_n\lmod$, we have that $$
\wf_i\wf_j M\cong\wf_j\wf_i M,\quad \we_i\we_j M\cong\we_j\we_i M,\quad \we_i\wf_j M\cong\wf_j\we_i M,\quad \eps_i(\wf_j M)=\eps_i(M).
$$
\end{lem}

\begin{proof} Since $i-j\neq 0,\pm1$, $L(i,j)\cong L(i)\circ L(j)\cong L(j)\circ L(i)\cong L(j,i)$ by the non-degenerate version of \cite[Theorem 6.1.4]{Klesh:book}. It is easy to see that $L(i,j)$ is a real simple module (cf. Definition \ref{realsim}). Applying Lemma \ref{real3properties}, we get that $$\begin{aligned}
\wf_i\wf_j M&=\bigl(M\triangledown L(j)\bigr)\triangledown L(i)\cong M\triangledown \bigl(L(j)\circ L(i)\bigr)\cong M\triangledown L(j,i)\cong M\triangledown L(i,j)\\
&\cong M\triangledown \bigl(L(i)\circ L(j)\bigr)\cong \bigl(M\triangledown L(i)\bigr)\triangledown L(j)=\wf_j\wf_i M .
\end{aligned}$$
As a result, we see that if $N = \we_i M\neq 0$ then $$
\we_i\wf_j M\cong\we_i\wf_j\wf_i N\cong\we_i\wf_i\wf_j N\cong\wf_j N\cong\wf_j\we_i M;
$$
If $\we_i M = 0$, $\we_i\wf_j M = 0$ by the shuffle lemma \cite[Lemma 2.4]{GV}. In a word, we have $\we_i\wf_jM\cong \wf_j\we_iM$ for any simple module $M$. Similarly, $\we_j\wf_iM\cong\wf_i\we_jM$ for any simple module $M$. As a consequence, $\eps_i(\wf_j M)=\eps_i(M)$.

It remains to show $\we_j\we_i M\cong \we_i\we_j M$. Assume first that $\we_i M\neq 0 \neq \we_j M$, then $\eps_i(M) = \eps_i(\we_j M) > 0$ by the last paragraph, and hence $\we_i\we_j M \neq 0$. In a similar way we show that $\we_j\we_i M \neq 0$. Thus in this case we have
$$\wf_i\wf_j(\we_j\we_i M) \cong M \cong \wf_j\wf_i(\we_i\we_j M) \cong \wf_i\wf_j(\we_i\we_j M)$$
which implies (by \cite{G} and \cite[Corollary 5.2.4]{Klesh:book}) that $\we_i\we_j M\cong\we_j\we_i M$. And if either $\we_iM = 0$ or $\we_jM = 0$, then it is easy to see that $\we_i\we_j M = 0 = \we_j\we_i M$ by the shuffle lemma \cite[Lemma 2.4]{GV}.
This completes the proof of the lemma.
\end{proof}

\medskip

From now on and until the end of this section, we assume that $1<\ell\in\N$, $e=2\ell$, $I:=\Z/2\ell\Z$, and $q:=\xi$ is a primitive $2\ell$th root of unity in $F$. To simplify notations, for any $i\in I$ and $j\in\Z$, we shall often write $i+j\in I$ instead of $i+j+2\ell\Z\in I$.
\medskip

\begin{dfn}\label{2Steps} Let $i\in I$ and set $\wi:=(i,i+\ell)\in I^2$. If $n\geq 2$ then we define the two-steps restriction functor $$
e_{\wi}:=e_i e_{i+\ell}:\,\, \Rep_I\HH_n\rightarrow \Rep_I\HH_{n-2} .
$$
\end{dfn}


For any $\Sym_n$-orbit $\gamma=\Sym_n\cdot (\xi^{i_1},\cdots,\xi^{i_n})$ of $\Sym_n$ on $(\xi^I)^n$, it determines a central character $$\begin{aligned}
\chi_\gamma:\,\, Z(\HH_n)&\rightarrow F,\\
f(X_1,\cdots,X_n)&\mapsto f(\xi^{i_1},\cdots,\xi^{i_n}),\quad\forall\,f(X_1,\cdots,X_n)\in K[X_1^{\pm 1},\cdots,X_n^{\pm 1}]^{\Sym_n}=Z(\HH_n) ,
\end{aligned}
$$
where $Z(\HH_n)$ denotes the center of $\HH_n$. We use $(\HH_n\lmod)[\gamma]$ to denote the block of $\Rep_I(\HH_n)$ corresponding to the central character $\chi_\gamma$ determined by $\gamma$.
This is the full subcategory of $\Rep_I(\HH_n)$ consisting of $M$ with $\charac M\in\bigoplus_{(\xi^{a_1},\cdots,\xi^{a_n})\in \gamma}\Z[a_1,\cdots,a_n]$, see \cite[(4.3)]{Klesh:book} for the corresponding definition in the degenerate setting.

\begin{dfn}\label{def:conseq} For each $i\in I$ and $m\in\N$, we define $$
\wi^m:=(\underbrace{\wi,\wi,\cdots,\wi}_{\text{$m$ copies}})=(\underbrace{i,i+\ell,i,i+\ell,\cdots,i,i+\ell}_{\text{$2m$ terms}}).
$$
and $$
\gamma_{i,m}:=\Sym_{2m}\cdot (\underbrace{\xi^i,-\xi^i,\xi^i,-\xi^i,\cdots,\xi^i,-\xi^i}_{\text{$2m$ terms}}) .
$$
\end{dfn}

\begin{lem} \label{uniquesimple} Let $i\in I$ and $n\in\N$. For any permutation $(i_1,\dots,i_{2n})$ of $\wi^n$, $$
L(\wi^n)\cong L(i_1)\circ L(i_2)\circ \cdots\circ  L(i_{2n})\cong L(i^n)\circ L((i+\ell)^n)
$$
is the unique simple module in $(\HH_{2n}\lmod)[\gamma_{i,n}]$.
\end{lem}

\begin{proof} Since $\ell>1$, by the non-degenerate versions of \cite[Lemmas 6.1.1, 6.1.2]{Klesh:book} we know that \begin{equation}\label{commu1}
L(i,i+\ell)\cong L(i)\circ  L(i+\ell)\cong L(i+\ell)\circ  L(i)\cong L(i+\ell,i) ,
\end{equation}
is simple.

By the same reasoning and using (\ref{commu1}) and the transitivity of induction functors, we can deduce that for any permutation  $(i_1,\dots,i_{2n})$ of $\wi^n$, $$
L(i_1)\circ  L(i_2)\circ \cdots\circ  L(i_{2n})\cong \underbrace{L(i)\circ  L(i+\ell)}\circ \cdots\circ  \underbrace{L(i)\circ  L(i+\ell)}=\ind_{2^n}^{2n}(\wi^n) .
$$
On the other hand, $$\begin{aligned}
\ind_{(2^n)}^{2n}(\wi^n)&=L(i)\circ  L(i+\ell)\circ \cdots\circ  L(i)\circ  L(i+\ell)\\
&\cong \underbrace{L(i)\circ L(i)\circ \cdots\circ  L(i)}_{\text{$n$ copies}}\circ \underbrace{L(i+\ell)\circ  L(i+\ell)\circ \cdots\circ  L(i+\ell)}_{\text{$n$ copies}}\\
&\cong\bigl(\underbrace{L(i)\circ  L(i)\circ \cdots\circ  L(i)}_{\text{$n$ copies}}\bigr)\circ
\bigl(\underbrace{L(i+\ell)\circ  L(i+\ell)\circ \cdots\circ  L(i+\ell)}_{\text{$n$ copies}}\bigr)\\
&\cong L(i^n)\circ  L((i+\ell)^n).
\end{aligned}$$
which is simple as $L(i^n)$ and $L((i+\ell)^n)$ are the classical Kato modules (\cite[Proposition 3.3]{GV}) of $\HH_n$. This shows that $L(\wi^n)\cong L(i_1)\circ  L(i_2)\circ \cdots\circ  L(i_{2n})$. Finally, the uniqueness follows from \cite[Corollary 5.7]{G} and the non-degenerate versions of \cite[Lemmas 6.1.4]{Klesh:book}.
\end{proof}

\begin{dfn}\label{GKato} Suppose $\ell>1$.  Let $i\in I$ and $n\in\N$. We call the $\HH_{2n}$-module $L(\wi^n)$ a generalized Kato module.
\end{dfn}

Let $\mu=(\mu_1,\mu_2,\cdots,\mu_r)$ be a composition of $n$. We define $2\mu:=(2\mu_1,2\mu_2,\cdots,2\mu_r)$ which is a composition of $2n$.
We use $\pi_{\mu,i}$ to denote the projection from $\Rep_I\HH_{2\mu}$ onto the block of $\Rep_I\HH_{2\mu}$ corresponding to the $\Sym_{2\mu}$-orbit of $$
(\underbrace{\xi^i,\cdots,\xi^i}_{\text{$\mu_1$ copies}},\underbrace{-\xi^i,\cdots,-\xi^i}_{\text{$\mu_1$ copies}},\cdots,\underbrace{\xi^i,\cdots,\xi^i}_{\text{$\mu_r$ copies}},\underbrace{-\xi^i,\cdots,-\xi^i}_{\text{$\mu_r$ copies}}).
$$

\begin{lem}\label{irre1} Let $i\in I$, $n\in\N$ and $\mu=(\mu_1,\mu_2,\cdots,\mu_r)$ be a composition of $n$. Suppose $\ell>1$. In the Grothendieck group of $\Rep_I\HH_{2\mu}$, $[\pi_{\mu,i}\res^{2n}_{2\mu}L(\wi^n)]=s[L(\wi^{\mu_1})\boxtimes\cdots\boxtimes L(\wi^{\mu_r})]$ for some integer $s>0$, and $\soc\pi_{\mu,i}\res^{2n}_{2\mu}L(\wi^n)$ is simple;
\end{lem}

\begin{proof} The first equality follows from Lemma \ref{uniquesimple}. Thus, any simple submodule of $\pi_{\mu,i}\res^{2n}_{2\mu}L(\wi^n)$ must be isomorphic to $L(\wi^{\mu_1})\boxtimes\cdots\boxtimes L(\wi^{\mu_r})$.

By the transitivity of the induction functor $\ind$, we know that $\ind_{2\mu}^{2n}L(\wi^{\mu_1})\boxtimes\cdots\boxtimes L(\wi^{\mu_r})\cong L(\wi^n)$.
Using Frobenius reciprocity, we can deduce that $$\begin{aligned}
0&<\dim\Hom_{\HH_{2\mu}}\bigl(L(\wi^{\mu_1})\boxtimes\cdots\boxtimes L(\wi^{\mu_r}),\pi_{\mu,i}\res^{2n}_{2\mu}L(\wi^n)\bigr)\\
&=\dim\Hom_{\HH_{2\mu}}\bigl(L(\wi^{\mu_1})\boxtimes\cdots\boxtimes L(\wi^{\mu_r}),\res^{2n}_{2\mu}L(\wi^n)\bigr)\\
&=\dim\Hom_{\HH_{2n}}\bigl(L(\wi^{\mu_1})\circ\cdots\circ L(\wi^{\mu_r}),L(\wi^n)\bigr)=1 .
\end{aligned}
$$
It follows that $\dim\Hom_{\HH_{2\mu}}\bigl(L(\wi^{\mu_1})\boxtimes\cdots\boxtimes L(\wi^{\mu_r}),\pi_{\mu,i}\res^{2n}_{2\mu}L(\wi^n)\bigr)=1$. Thus, $\soc\pi_{\mu,i}\res^{2n}_{2\mu}L(\wi^n)\cong L(\wi^{\mu_{1}}) \boxtimes \cdots \boxtimes L(\wi^{\mu_{r}})$ is simple.
\end{proof}

\begin{dfn} Let $i\in I$ and $m\in\N$. We define $$
\whd_{\wi^m}:=\bigoplus_{\substack{\ua=(a_1,\cdots,a_{2m})\in I^{2m}\\
(\xi^{a_1},\cdots,\xi^{a_{2m}})\in\gamma_{i,m}}}\Delta_{\ua} .
$$
\end{dfn}

It is clear that for any $M\in\HH_{n}\lmod$, $\whd_{\wi^m}M$ is an $\HH_{n-2m,2m}$-submodule of $\res_{n-2m,2m}^{n}M$. So $\whd_{\wi^m}$ does define a functor
$\Rep_I\HH_{n}\rightarrow\Rep_I\HH_{n-2m,2m}$. We have a functorial isomorphism: \begin{equation}\label{functorial}
\Hom_{\HH_{n-2m,2m}}(N\boxtimes L(\wi^m),\whd_{\wi^m}M)\cong\Hom_{\HH_n}\bigl(N\circ L(\wi^m), M\bigr) .
\end{equation}

\begin{lem} \label{DeltaCharacter} Let $i\in I$, $m\in\N$ with $1\leq m\leq n/2$. Let $M\in\HH_{n}\lmod$. Then $\whd_{\wi^m}M$ is the largest submodule
of $\res_{n-2m,2m}^{n}M$ such that all of its composition factors are of the form $N\boxtimes L(\wi^m)$ for some simple module $N\in\HH_{n-2m}\lmod$.
\end{lem}

\begin{proof} This follows from Lemma \ref{uniquesimple}.
\end{proof}

Let $M\in\Rep_I\HH_{n}$ and $0\leq m\leq n/2$. We write $\charac M=\sum_{\ua\in I^{n}}r_{\ua}[a_1,\cdots,a_{n}]$. Then $$\begin{aligned}
&\charac\whd_{\wi^m} M=\sum_{\substack{\ua=(a_1,\cdots,a_{n})\in I^{n}\\ (\xi^{a_{n-2m+1}},\cdots,\xi^{a_{n}})\in\gamma_{i,m}}}r_{\ua}[a_1,\cdots,a_{n}] .
\end{aligned}
$$

\begin{dfn}\label{epsHat1} Let $M$ be a simple module in $\Rep_I\HH_{n}$. Let $i\in I$. We define $$
{\eps}_{\widehat{i}}(M):=\max\{m\geq 0|\whd_{\wi^m}M\neq 0\}.
$$
\end{dfn}

\begin{cor}\label{epsHat2} Let $M$ be a simple module in $\Rep_I\HH_{n}$. Let $i\in I$. Then $$
{\eps}_{\widehat{i}}(M)=\max\{m\geq 0|e_{\wi}^mM\neq 0\}. $$
\end{cor}

\begin{proof} If $e_{\wi^m}M\neq 0$ then it is clear that $\whd_{\wi^m}M\neq 0$. Conversely, if $\whd_{\wi^m}M\neq 0$, then it follows from Lemma \ref{DeltaCharacter} and Lemma \ref{uniquesimple} that $e_{\wi}^mM\neq 0$. Hence the corollary follows.
\end{proof}

\begin{lem}\label{epswi01} Let $M$ be a simple module in $\Rep_I\HH_{n}$. Let $i\in I$, $\eps:={\eps}_{\widehat{i}}(M)$. If $N\boxtimes L(\wi^m)$ is a simple submodule of $\whd_{\wi^m} M$ for some $0\leq m\leq\eps$, then ${\eps}_{\widehat{i}}(N)=\eps-m$.
\end{lem}

\begin{proof} By the definition of ${\eps}_{\widehat{i}}$ we see that ${\eps}_{\widehat{i}}(N)\leq\eps-m$. By (\ref{functorial}) we get a nonzero (and hence surjective) homomorphism $N\circ L(\wi^m)\rightarrow M$. So, by characters consideration and the Shuffle Lemma \cite[Lemma 2.4]{GV}, we can deduce that
$\eps_{\widehat{i}}(N)+m\geq\eps$. Hence ${\eps}_{\widehat{i}}(N)=\eps-m$.
\end{proof}

\bigskip
\section{Some real simple modules}

Throughout this section, unless otherwise specified, we assume that $2<e\in\N$, $I:=\Z/e\Z$, and $q$ is a primitive $e$th root of unity in $F$.

In \cite{KKKO}, Kang, Kashiwara et al. develop the theory of real simple modules for quantum affine algebras and for quiver Hecke algebras which is proved to be a very powerful tool. Rouquier has first presented an isomorphism in \cite[Proposition 3.18]{Rou1} between certain localized forms of affine Hecke algebras and of quiver Hecke algebras. It follows from the isomorphism that there is an equivalence (\cite[Theorem 3.19]{Rou1}) between the category of finite dimensional integral representations of the affine Hecke algebras of type $A$ and the category of finite dimensional representations of the quiver Hecke algebras of type $A$.

Later, using a lift of Brundan-Kleshchev's isomorphism \cite{BK:GradedKL} between cyclotomic Hecke algebras of type $A$ and cyclotomic quiver Hecke algebras of type $A$ to their affine counterparts, the second author of this paper and Fang Li proved in \cite[\S5]{HL} an isomorphism between some generalized Ore localizations of some modified affine Hecke algebras and of the quiver Hecke algebras of type $A$. The equivalence between these two finite dimensional modules categories induced by the isomorphisms is compatible with the convolution products on both sides.

We mimic \cite{KKKO} to give the definition of real simple module in the category of finite dimensional modules over $\HH_n$.

\begin{dfn}\label{realsim} Let $M\in\Rep_I\HH_n$ be a simple module. If $M\circ M$ is simple $\HH_n$-module, then we call $M$ a real simple $\HH_n$-module.
\end{dfn}

\begin{exmp}\label{Katoreal} For each $i\in I$ and $n\in\N$, the Kato module $L(i^n)$ is a real simple $\HH_n$-module. In fact, this is clear because $L(i^n)\circ L(i^n)\cong L(i^{2n})$ is again a simple (Kato) module.
\end{exmp}

\begin{lem}\label{real3properties} Let $M\in\Rep_I\HH_n$ be a real simple module, $N\in\Rep_I\HH_m$ be a simple module.

1) $M\circ N$ has a simple head and a simple socle. Similarly, $N\circ M$ has a simple head and a simple socle;

2) if $M\circ N\cong N\circ M$, then $M\circ N$ is a simple $\HH_{m+n}$-module. Conversely, if $M\circ N$ is simple, then $M\circ N\cong N\circ M$;

3) if $N$ is a real simple module too and $M\circ N\cong N\circ M$, then $M\circ N$ is a real simple $\HH_{m+n}$-module. In particular, for any $k\geq 1$, $M^{\circ k}$ is a real simple $\HH_{nk}$-module;

4) if $M\triangledown N\cong N\triangledown M$, then $M\circ N$ is simple.
\end{lem}

\begin{proof} For 1) and 2), they follow from \cite[Theorem 3.2, Corollaries 3.3, 3.4]{KKKO} and \cite[(5.10)]{HL}. 4) follows from \cite[Corollary 3.9]{KKKO} because the socle of $M\circ N$ is isomorphic to the head of $(M\circ N)^*\cong N^*\circ M^*\cong N\circ M$. Finally, 3) is a consequence of 2).
\end{proof}

We remark that the notion of real simple modules can also be defined for the category of finite dimensional modules over the degenerate affine Hecke algebra of type $A$ in a similar way, and the above lemma also holds in the context of the degenerate affine Hecke algebra of type $A$ as the main results in \cite{HL} work for both the non-degenerate and the degenerate affine Hecke algebras of type $A$.

\medskip
For the rest of this section, we fix $i,j\in I$.

\begin{lem}\label{real1} The module $L(i,j)$ is a real simple $\HH_2$-module and $L(i,j)\circ L(i,j)\cong L(i^2,j^2)$.
\end{lem}

\begin{proof} Suppose $i-j\not\in\{\pm1\}$ then by the non-degenerate version of \cite[Theorem 6.1.4]{Klesh:book}, $L(i)\circ L(j)\cong L(j)\circ L(i)$ is simple. Since $L(i)$ and $L(j)$ are real simple modules, it follows from Lemma \ref{real3properties} that both $L(i)\circ L(j)$ and $L(j)\circ L(i)$ have unique simple heads and hence $L(i,j)\cong L(i)\circ L(j)\cong L(j)\circ L(i)\cong L(j,i)$, from which the lemma follows at once.

Now suppose $i-j\in\{\pm1\}$. Set $L:=L(i,j)\circ L(i,j)$. Recall that $\charac L(i,j)=[i,j]$. By the Shuffle Lemma \cite[Lemma 2.4]{GV}, \begin{equation}\label{chL} \charac(L)=4[i^2,j^2]+2[i,j,i,j]. \end{equation}
We set $\gamma:=\Sym_4\cdot (i,j,i,j)$. By the non-degenerate version of \cite[Lemma 6.2.2]{Klesh:book}, $$
\we_j L(i^2,j^2)\cong L(i^2,j)\cong  L(i,j,i)\cong \we_j L(i,j,i,j).
$$
It follows from \cite[Corollary 3.6]{GV} that $L(i^2,j^2)\cong L(i,j,i,j)$. Similarly, $L(j^2,i^2)\cong L(j,i,j,i)$.

The block subcategory $(\HH_4\lmod)[\gamma]$ has at most $4$ distinct isoclasses of simple modules. Namely, $$
[L(i^2,j^2)]=[L(i,j,i,j)],\quad [L(j^2,i^2)]=[L(j,i,j,i)],\quad [L(i,j,j,i)],\quad [L(j,i,i,j)] .
$$
By the non-degenerate version of \cite[Lemma 6.1.1]{Klesh:book}, $[i_1,i_2,i_3,i_4]$ always appears nonzero coefficient in the character of $\charac L(i_1,i_2,i_3,i_4)$. By characters consideration and the Shuffle Lemma \cite[Lemma 2.4]{GV}, we see that $L(i,j,i,j)\cong L(i^2,j^2)$ is the only composition factor of $L$. Furthermore, $[L:L(i,j,i,j)]\in\{1,2\}$.

Suppose that $[L:L(i,j,i,j)]=2$. Then we have a short exact sequence $$
0\rightarrow L(i^2,j^2)\rightarrow L\rightarrow L(i^2,j^2)\rightarrow 0 .
$$
Applying $e_j$ on the above exact sequence, we get a new exact sequence \begin{equation}\label{ejL}
0\rightarrow e_j L(i^2,j^2)\rightarrow e_j L\rightarrow e_j L(i^2,j^2)\rightarrow 0 .
\end{equation}
It follows that $$
\charac e_j L=4[i^2,j]+2[i,j,i]=2\charac e_j L(i^2,j^2) .
$$

On the other hand, by \cite[Theorem 9.13]{G},  $$
\charac (e_j L(i^2,j^2))=2\charac (\we_j L(i^2,j^2))+\sum_{\alpha}c_\alpha\charac M_\alpha=2\charac (L(i^2,j))+\sum_{\alpha}c_\alpha\charac M_\alpha ,
$$
where for each $\alpha$, $M_\alpha\in\HH_3\lmod$ is simple, $c_\alpha\in\N$ and $\eps_j(M_\alpha)=0$. However, by the above calculation, (\ref{chL}) and (\ref{ejL}), we have that $$\begin{aligned}
\charac (e_j L(i^2,j^2))&=\frac{1}{2}\charac(e_jL)=2[i^2,j]+[i,j,i],\\
2\charac (L(i^2,j))&=2(2[i^2,j]+[i,j,i])=4[i^2,j]+2[i,j,i] .
\end{aligned}
$$
We get a contradiction! This proves that
$L\cong L(i^2,j^2)$ is simple. Hence $L(i,j)$ is a real simple module.
\end{proof}

\begin{lem}\label{ijcommutes} For any $m,k\in\N$, $L(i,j)^{\circ k}\circ L(i^m)\cong L(i^m)\circ L(i,j)^{\circ k}$ is simple, and $L(j,i)^{\circ k}\circ L(i^m)\cong L(i^m)\circ L(j,i)^{\circ k}$ is simple.
\end{lem}

\begin{proof} If $i-j\not\in\{\pm1\}$ then the lemma clearly holds because  $L(i^m)\cong L(i)^{\circ m}$ and in that case $L(i)\circ L(j)\cong L(j)\circ L(i)\cong L(i,j)$ is simple by the non-degenerate version of \cite[Theorem 6.1.4]{Klesh:book}. It remains to consider the case when $i-j\in\{\pm1\}$.

In this case, by Lemma \ref{real3properties}. 3),4) and Lemma \ref{real1}., it suffices to prove the lemma for $k=1=m$ as $L(i^m)\cong L(i)^{\circ m}$ and both $L(i)$ and $L(i,j)$ are real simple modules. By the non-degenerated version of \cite[Lemmas 6.2.2]{Klesh:book}, we can get that $$
L(i,j,i)\cong L(i,j)\circ L(i)\cong L(i)\circ L(i,j),\quad L(j,i^2)\cong L(j,i)\circ L(i)\cong L(i)\circ L(j,i),
$$
which are both simple by Lemma \ref{real3properties}. 3). This proves the lemma.
\end{proof}

\begin{lem} \label{iji1} Set $\gamma_n:=\Sym_n\cdot (\underbrace{\xi^i,\cdots,\xi^i}_{\text{$n-1$ copies}},\xi^j)$, where $2\leq n\in\N$. Suppose $i-j\in\{\pm1\}$. The block subcategory $(\HH_n\lmod)[\gamma_n]$ has only two distinct isoclasses of simple modules. namely, $[L(i,j,i^{n-2})]$ and $[L(j,i^{n-1})]$. Moreover, $L(i^r,j,i^s)\cong L(i,j,i^{r+s-1})$ for any $r\geq 1,s\geq 0$.
\end{lem}

\begin{proof} By \cite[Proposition 3.3]{GV} and the transitivity of induction functors, $$
L(i,j)\circ L(i^{n-2})\cong L(i,j)\circ \underbrace{L(i)\circ\cdots\circ L(i)}_{\text{$n-2$ copies}}
$$
By Lemma \ref{ijcommutes}, $L(i,j)\circ L(i^k)$ is simple for any $k\geq 0$. It follows that $$
L(i,j)\circ L(i^{n-2})\cong L(i,j)\circ L(i)^{\circ n-2}\cong \widetilde{f}_i^{n-2}L(i,j)=L(i,j,i^{n-2}) .
$$
Similarly, we have $L(j,i)\circ L(i^{n-2})\cong L(j,i^{n-1})$.

By Lemma \ref{ijcommutes}, $L:=L(i^{n-2})\circ L(i,j)$ is simple, so by Shuffle Lemma \cite[Lemma 2.4]{GV} $\eps_j(L) = 1$ and hence $\we_jL \neq 0$ is forced to be isomorphic to $L(i^{n-1})$. Thus $L\cong L(i^{n-1},j)$ by \cite[Corollary 3.6]{GV}. Applying Lemma \ref{ijcommutes}, $L\cong L(i,j,i^{n-2})$.

In general, for any $r\geq 1, s\geq 0$, $L(i^r,j,i^s)\cong\head\bigl(L(i^r,j)\circ L(i^s)\bigr)$ by \cite[Corollary 3.6(i)]{GV}. Therefore, by Lemma \ref{ijcommutes} and the result we obtained in the last paragraph, $$\begin{aligned}
L(i^r,j)\circ L(i^s)&\cong L(i^{r-1})\circ L(i,j)\circ L(i^s)\cong L(i,j)\circ L(i^{r-1})\circ L(i^s)\cong L(i,j)\circ L(i^{r+s-1})\\
&\cong L(i,j,i^{r+s-1}).
\end{aligned}
$$
This proves the second part of the lemma. Finally, note that $$
\we_i^{n-1}L(j,i^{n-1})\cong L(j)\neq 0,\quad \we_i^{n-1}L(i,j,i^{n-2})=\we_iL(i,j)=0 .
$$
We see that $L(j,i^{n-1})\not\cong L(i,j,i^{n-2})$. Now the first part of the lemma follows from the second part and the result we obtained in the last paragraph.
\end{proof}

\begin{cor}\label{convolution2} Let $m\in\N$. Then $L(i^2,j^2)\circ L(i^m)\cong L(i^m)\circ L(i^2,j^2)$ is simple, and $L(j^2,i^2)\circ L(i^m)\cong L(i^m)\circ L(j^2,i^2)$ is simple.
\end{cor}

\begin{proof} This follows from Lemma \ref{ijcommutes} and Lemma \ref{real1}.
\end{proof}

\begin{cor}\label{realCor} Let $m,k\in\N$. Then both $L(i^m)\circ L(i,j)^{\circ k}$ and $L(i^m)\circ L(j,i)^{\circ k}$ are real simple $\HH_{m+2k}$-modules.
\end{cor}

\begin{proof} By Lemma \ref{ijcommutes}, $L(i^m)\circ L(i,j)^{\circ k}\cong L(i,j)^{\circ k}\circ L(i^m)$ and $L(i^m)\circ L(j,i)^{\circ k}\cong L(j,i)^{\circ k}\circ L(i^m)$. Since both $L(i^m)\cong L(i)^{\circ m}$, $L(i,j)^{\circ k}$ and $L(j,i)^{\circ k}$ are real simple modules, it follows from Lemma \ref{real3properties} 4) that both $L(i^m)\circ L(i,j)^{\circ k}$ and $L(i^m)\circ L(j,i)^{\circ k}$ are real simple $\HH_{m+2k}$-modules.
\end{proof}

\begin{cor}\label{headCor} Let $m,k\in\N$. For any simple module $M\in\HH_n\lmod$, we have that both $\head\bigl(M\circ L(i^m)\circ L(i,j)^{\circ k}\bigr)$ and $\head\bigl(M\circ L(i^m)\circ L(j,i)^{\circ k}\bigr)$ are simple.
\end{cor}

\begin{proof} This follows from Lemma \ref{real3properties} 2) and Corollary \ref{realCor}.
\end{proof}

We remark that the above corollary for the degenerate affine Hecke algebras also holds by the same argument. In particular, we partially fixes a gap of \cite[Lemma 6.3.2]{Klesh:book}, see \cite{Klesh:book2}. To be more precise, we prove that \cite[Lemma 6.3.2]{Klesh:book} holds whenever either $p>2$ in the degenerate affine Hecke algebra case or $e>2$ in the non-degenerate affine Hecke algebra case.
\medskip

\bigskip
\section{The operators $\we_{\widehat{i}}, \wf_{\wi}$ and their properties}

In this section, we assume that $1<\ell\in\N$, $q:=\xi$ is a primitive $2\ell$th root of unity in $F$. In particular, this implies that $\cha F\neq 2$.

\begin{lem}\label{efcommutes} Let $M$ be a simple module in $\Rep_I\HH_{n}$ and $i\in I$. Then $\widetilde{e}_i\widetilde{e}_{i+\ell}M\cong \widetilde{e}_{i+\ell}\widetilde{e}_{i}M$, and $\widetilde{e}_i\widetilde{e}_{i+\ell}M\neq 0$ if and only if $e_{\wi}M\neq 0$.
\end{lem}

\begin{proof} Since $\ell>1$, it follows from Lemma \ref{eijcommutes} that $\widetilde{e}_i\widetilde{e}_{i+\ell}M\cong \widetilde{e}_{i+\ell}\widetilde{e}_{i}M$.

Suppose that $e_{\wi}M\neq 0$. By Lemma \ref{2functorsIso}, we have that $e_iM\neq 0\neq e_{i+\ell}M$. Hence $\widetilde{e}_{i}M\neq 0\neq \widetilde{e}_{i+\ell}(M)$. This implies that $\eps_i(M)\geq 1\leq\eps_{i+\ell}(M)$. We set $N:=\widetilde{e}_{i+\ell}M\neq 0$. Applying Lemma \ref{eijcommutes}, we can deduce that $\eps_i(N)=\eps_i(M)\geq 1$. Thus $\widetilde{e}_iN=\widetilde{e}_i\we_{i+\ell}M\neq 0$ as required.

Conversely, if $\widetilde{e}_i\widetilde{e}_{i+\ell}M\neq 0$ then $e_i{\we}_{i+\ell}M\neq 0$, and hence $e_{\wi}M={e}_i{e}_{i+\ell}M\neq 0$. This completes the proof of the lemma.
\end{proof}

Recall from Definition \ref{epsHat1} and Corollary \ref{epsHat2} that $$
{\eps}_{\widehat{i}}(M)=\max\{m\geq 0|\whd_{\wi^m}M\neq 0\}=\max\{m\geq 0|e_{\wi}^mM\neq 0\}.
$$

\begin{lem}\label{imsoclehead} Let $m\in\N$, $i\in I$ and $M\in\Rep_I\HH_n$ be a simple module. Then $\head(M\circ L(\wi^m))\cong(\widetilde{f}_i\widetilde{f}_{i+\ell})^mM$.
\end{lem}

\begin{proof} Since $L(\wi)\cong L(i)\circ L(i+\ell)\cong L(i+\ell)\circ L(i)$, it is easy to see that $(\widetilde{f}_i\widetilde{f}_{i+\ell})^mM$ is a quotient of
$M\circ L(\wi^m)$. Using Lemma \ref{real1}, we know that $L(\wi^m)$ is a real simple module. Now the lemma follows from Lemma \ref{real3properties}.
\end{proof}

\begin{prop}\label{RedBao} Let $M$ be a simple module in $\Rep_I\HH_n$. Let $i\in I$. If $e_{\wi}M\neq 0$ then $\soc e_{\wi}M$ is simple. Furthermore, $\soc e_{\wi}M\cong \widetilde{e}_i\widetilde{e}_{i+\ell}M$ and ${\eps}_{\widehat{i}}(\soc e_{\wi}M)={\eps}_{\widehat{i}}(M)-1$.
\end{prop}


\begin{proof} Being a center element of $\HH_n$, $X_1+\cdots+X_n$ acts as a scalar $c$ on $M$. Similarly, the center element $\sum_{1\leq i<j\leq n}X_iX_j$ acts as a scalar $c'$ on $M$.

Assume that $e_{\wi}M\neq 0$. Let $L\subseteq e_{\wi}M$ be any simple $\HH_{n-2}$-submodule. Then the center element $X_1+\cdots+X_{n-2}$ of $\HH_{n-2}$ acts as the scalar $c-q^i-q^{i+\ell}=c$ on $L$. Similarly, the center element $\sum_{1\leq i<j\leq n-2}X_iX_j$ of $\HH_{n-2}$ acts as a scalar on $L$. Since, $$\begin{aligned}
&X_{n-1}+X_n=(X_1+\cdots+X_n)-(X_1+\cdots+X_{n-2}),\\
&X_{n-1}X_n=\sum_{1\leq i<j\leq n}X_iX_j-\sum_{1\leq i<j\leq n-2}X_iX_j-(X_1+\cdots+X_{n-2})(X_{n-1}+X_n),
\end{aligned}
$$
it follows that both $X_{n-1}+X_n$ and $X_{n-1}X_n$ act as scalars on $L$ and these scalars are invariant when $L$ varies by block consideration which implies that both $X_{n-1}+X_n$ and $X_{n-1}X_n$ act as scalars on $\soc e_{\wi}M$. Note that both $X_{n-1}$ and $X_n$ stabilize $\soc e_{\wi}M$, and $q^i$ is the only eigenvalue of $X_{n-1}$ on $\soc e_{\wi}M$, $-q^i=q^{i+\ell}$ is the only eigenvalue of $X_{n}$ on $\soc e_{\wi}M$, it follows that $X_{n-1}+X_n$ act as $0$ on $\soc e_{\wi}M$ and $X_{n-1}X_n$ act as $-q^{2i}$ on $\soc e_{\wi}M$.

We claim that both $X_{n-1}$ and $X_n$ act as scalars on $\soc e_{\wi}M$. Since $X_n$ commutes with $X_{n-1}$, they can be (upper)-triangularized on $\soc e_{\wi}M$ simultaneously.  Without loss of generality, we can assume that there is a basis of $\soc e_{\wi}M$ under which the matrix of $X_{n-1}$ is a Jordan normal form $A$ with diagonal elements all being equal to $q^i$, and under which the matrix of $X_n$ is a upper-triangular matrix $B$ with diagonal elements all being equal to $-q^i=q^{i+\ell}$, and $A+B=0, AB=-q^{2i}$. If either $A$ or $B$ is a diagonal matrix, then it is immediate that both $A$ and $B$ are diagonal matrices and our claim follows. Suppose that this is not the case. Since $A$ is a non-diagonal Jordan matrix, $B$ is upper-triangular and $A+B=0$, it is easy to see that if $A(i,i+1)=1$ then $B(i,i+1)=-1$ and hence $(AB)(i,i+1)=-2q^i\neq 0$ as $\cha F\neq 2$, a contradiction to the fact that $AB=-q^{2i}$. This proves our claim.

Therefore, $X_{n-1}$ acts as $q^i$ on $\soc e_{\wi}M$, and $X_n$ act as $-q^i$ on $\soc e_{\wi}M$. In particular, any constituent $L$ of $\soc e_{\wi}M$ contributes a simple submodule of $\res_{n-2,1,1}^{n}M$ which is isomorphic to $L\boxtimes L(i)\boxtimes L(i+\ell)$. By Frobenius reciprocity, we have a surjective homomorphism $$
\ind_{n-2,2}^{n}L\boxtimes L(i,i+\ell)\cong\ind_{n-2,1,1}^{n}L\boxtimes L(i)\boxtimes L(i+\ell)\twoheadrightarrow M .
$$
Since $\ell>1$, $L(i,i+\ell)$ is a real simple module by Lemma \ref{real1}. It follows from Lemma \ref{real3properties} that $L\circ L(i,i+\ell)$ has a unique simple head. On the other hand, by definition and (\ref{widetildeef}), there is a natural surjection: $$
\ind_{n-2,1,1}^{n}L\boxtimes L(i)\boxtimes L(i+\ell)\twoheadrightarrow\widetilde{f}_{i+\ell}\widetilde{f}_iL\neq 0 .
$$
It follows that $M\cong\head(\ind_{n-2,2}^{n}L\boxtimes L(i,i+\ell))\cong\widetilde{f}_{i+\ell}\widetilde{f}_iL$. Hence by \cite{G}, the non-degenerate version of \cite[Lemma 5.2.3]{Klesh:book} and Lemma \ref{efcommutes},
$L\cong\widetilde{e}_i\widetilde{e}_{i+\ell}M\cong \widetilde{e}_{i+\ell}\widetilde{e}_{i}M$.

Applying Frobenius reciprocity together with the proof in the above two paragraphs, we get that $$\begin{aligned}
&\quad\,\dim\Hom_{\HH_{n-2}}(\widetilde{e}_i\widetilde{e}_{i+\ell}M,e_ie_{i+\ell}M)=\dim\Hom_{\HH_{n-2,1,1}}(\widetilde{e}_i\widetilde{e}_{i+\ell}M\boxtimes L(i)\boxtimes L(i+\ell),\res_{n-2,1,1}^{n}M)\\
&=\dim\Hom_{\HH_n}\bigl(\ind_{n-2,1,1}^{n}\widetilde{e}_i\widetilde{e}_{i+\ell}M\boxtimes L(i)\boxtimes L(i+\ell), M\bigr)\\
&=\dim\Hom_{\HH_n}\bigl(\ind_{n-2,2}^{n}\widetilde{e}_i\widetilde{e}_{i+\ell}M\boxtimes L(\wi), M\bigr)\\
&=\dim\Hom_{\HH_n}\bigl(\ind_{n-2,2}^{n}L\boxtimes L(\wi), M\bigr)=1 .
\end{aligned}
$$
Thus $\soc e_{\wi}M\cong \widetilde{e}_i\widetilde{e}_{i+\ell}M$ is simple.

Apply Lemma \ref{epswi01} to the case $m = 1$, we get ${\eps}_{\widehat{i}}(\soc e_{\wi}M)={\eps}_{\widehat{i}}(M)-1$.
\end{proof}

\begin{dfn} Let $i\in I$ and $M$ a simple module in $\Rep_I\HH_n$. We define $$
\widetilde{e}_{\wi}(M):=\soc\circ e_{\wi}(M),\quad\,\widetilde{f}_{\wi}(M):=\head\bigl(\ind_{n-2,2}^{n}M\boxtimes L(i,i+\ell)\bigr) .
$$
\end{dfn}

\begin{cor}\label{KashiwaraOper} Let $i\in I$ and $M$ a simple module in $\Rep_I\HH_n$. Then $$\begin{aligned}
 \widetilde{e}_{\wi}(M)&\cong \widetilde{e}_i\widetilde{e}_{i+\ell}M\cong \widetilde{e}_{i+\ell}\widetilde{e}_{i}M,\quad
\widetilde{f}_{\wi}(M)\cong \widetilde{f}_i\widetilde{f}_{i+\ell}M\cong \widetilde{f}_{i+\ell}\widetilde{f}_{i}M,\\
 {\eps}_{\widehat{i}}(M)&=\max\{m\geq 0|\widetilde{e}_{\wi}^mM\neq 0\}=\max\{m\geq 0|(\widetilde{e}_i\widetilde{e}_{i+\ell})^mM\neq 0\}\\
 &=\min\{\eps_i(M),\eps_{i+\ell}(M)\} \end{aligned}
$$
Furthermore, if $N\in \Rep_I\HH_{n+2}$ is a simple module, then $\widetilde{f}_{\wi}M\cong N$ if and only if $\widetilde{e}_{\wi}N\cong M$.
\end{cor}

By Corollary \ref{KashiwaraOper}, it is clear that for any $n\geq 1$ and any simple module $M\in\Rep_I\HH_n$, \begin{equation}\label{weEpsHat1}
\text{$0\neq \we_{\wi}M$ is simple if and only if $\eps_{\wi}(M)>0$.}
\end{equation}

The following lemma is a ``hat" analogue of \cite[Lemma 5.1.4, Theorem 5.1.6 and Lemma 5.2.1]{Klesh:book}. However, our proof is different. In fact, we remark that \cite[Theorem 5.1.6 and Lemma 5.2.1]{Klesh:book} and their proof can not be transformed directly into our ``hat" set-up because the naive``hat" version of \cite[Lemma 5.1.4]{Klesh:book} does not hold unless we have the extra assumption that $\eps_i(M)=\eps_{i+\ell}(M)$.

\begin{lem}\label{socWhd2} Let $M$ be a simple module in $\Rep_I\HH_{n}$. Let $i\in I$, $\eps:={\eps}_{\widehat{i}}(M)$.

1) For any $0\leq m\leq {\eps}$, $\soc\whd_{\wi^m}M\cong (\we_i\we_{i+\ell})^mM\boxtimes L(\wi^m)$ with ${\eps}_{\widehat{i}}((\we_i\we_{i+\ell})^mM)=\eps-m$.

2) If $\eps_i(M) = \eps_{i+\ell}(M)$, then $\whd_{\wi^\eps} M = \soc\whd_{\wi^\eps} M$;
\end{lem}


\begin{proof} 1)  Let $N\boxtimes L(\wi^m)$ be a simple submodule of $\whd_{\wi^m}M$, then by Frobenius reciprocity (\ref{functorial}) and Lemma \ref{imsoclehead}, $N \cong \we_i^m\we_{i+\ell}^mM$. Since $$\begin{aligned}
&\quad\,\dim\Hom_{\HH_{n-2m,2m}}(\we_i^m\we_{i+\ell}^mM\boxtimes L(\wi^m),\whd_{\wi^m}M)\\
&=\dim\Hom_{\HH_{n-2m,2m}}(\we_i^m\we_{i+\ell}^mM\boxtimes L(\wi^m),\res^n_{n-2m,m}M)\\
&=\dim\Hom_{\HH_{n}}((\we_i^m\we_{i+\ell}^mM)\circ L(\wi^m),M)\\
&=\dim\Hom_{\HH_{n}}((\we_i^m\we_{i+\ell}^mM)\triangledown L(\wi^m),M)=1 .  \qquad\text{(as $L(\wi^m)$ is real by Lemma \ref{real1})}
\end{aligned}
$$
This proves that $\soc\whd_{\wi^m}M$ is simple. The second part of 1) follows from Lemma \ref{epswi01}.

2) Suppose $\eps_i(M) = \eps_{i+\ell}(M)$. By Lemma \ref{epswi01}, we can assume $N\boxtimes L(\wi^\eps)$ is a simple submodule of $\whd_{\wi^\eps} M$, where $N\in\Rep_I(\HH_{n-2\eps})$ is simple with $\eps_{\wi}(N)=0$. Applying (\ref{functorial}) and Lemma \ref{imsoclehead}, we can deduce that $M$ is the unique simple head of $N\circ L(\wi^\eps)$ and $\eps_{i}(N)=\eps_i(M)-\eps=\eps_{i+\ell}(M)-\eps=\eps_{i+\ell}(N)$. In particular, $\eps_{i}(N)=0=\eps_{i+\ell}(N)$ by Corollary \ref{KashiwaraOper}, and $\whd_{\wi^\eps}(M)$ is a quotient of $\whd_{\wi^\eps}(N\circ L(\wi^\eps))$. Using the Shuffle Lemma \cite[Lemma 2.4]{GV} and the equality $\eps_{i}(N) = 0= \eps_{i+\ell}(N)$, we get that $\whd_{\wi^\eps}(N\circ L(\wi^\eps))\cong N\boxtimes L(\wi^\eps)$. Thus $\whd_{\wi^\eps} M\cong N\boxtimes L(\wi^\eps)$ is simple and $N\cong (\we_{i}\we_{i+\ell})^{\eps}M$ by the result proved in 1). This proves 2).
\end{proof}

\begin{cor}\label{HatKashiwaraOperators} Let $i,j\in I$ and $M$ be a simple module in $\Rep_I\HH_n$. Suppose $i\not\in\{j\pm 1, j+\ell\pm 1\}$. Then $$
\we_{\widehat{i}}\we_{\widehat{j}}M\cong \we_{\widehat{j}}\we_{\widehat{i}}M, \quad \wf_{\widehat{i}}\wf_{\widehat{j}}M\cong \wf_{\widehat{j}}\wf_{\widehat{i}}M, \quad \we_{\wi}\wf_{\wj} M\cong\wf_{\wj}\we_{\wi} M, \quad \eps_{\wi}(\wf_{\widehat{j}}M)=\eps_{\wi}(M).
$$
\end{cor}

\begin{proof} The first three isomorphisms follows from Lemma \ref{eijcommutes} and Corollary \ref{KashiwaraOper}. The fourth equality follows from the fourth equality in Lemma \ref{eijcommutes} by noting that $\eps_{\wi}(M)=\min\{\eps_i(M),\eps_{i+\ell}(M)\}$.
\end{proof}

\begin{lem} Let $i\in I, 0\leq m\leq n/2$. Suppose that $N\in\Rep_I\HH_{n-2m}$ is simple with $\eps=\eps_{\wi}(N)$. Set $M:=N\circ L(\wi^m)$.  Then

1) $K:=\head(M)$ is simple with $\eps_{\wi}(K):=m+\eps$;

2) If $\eps_i(N) = \eps_{i+\ell}(N)$, then all the other composition factors $L$ of $M$ have $\eps_{\wi}(L)<m+\eps$;
\end{lem}

\begin{proof} 1) By Lemma \ref{real1}, $L(\wi^m)$ is a real simple module. Applying Lemma \ref{real3properties}, we deduce that $K:=\head(M)$ is simple. Using Lemma \ref{imsoclehead} and Corollary \ref{KashiwaraOper}, we see that $K\cong\widetilde{f}_{\wi}^mN$ and $\eps_{\wi}(K):=m+\eps$.

2) By Corollary \ref{KashiwaraOper} and assumption, we have $\eps_i(N) = \eps_{i+\ell}(N)=\eps_{\wi}(N)=\eps$. Applying Lemma \ref{real3properties}, we get a natural surjection $$
M\cong N\circ L(i^m)\circ L((i+\ell)^m)\twoheadrightarrow (N\triangledown L(i^m))\triangledown L((i+\ell)^m)\cong\head(M)
$$
as $M$ has a unique simple head.

Applying Lemma \cite[Lemma 3.5]{GV}, we can deduce that $\eps_{i}(N\triangledown L(i^m))=\eps_{i}(\wf_i^mN)=\eps+m$, while all the other composition factors $L_1$ of $N\circ L(i^m)$ have $\eps_{i}(L_1)<\eps+m$. Hence for those $L_1$, $\eps_{\wi}(L')\leq \eps_i(L')<m+\eps$ where $L'$ is any composition factor of $L_1\circ L((i+\ell)^m)$. Now by Lemma \ref{eijcommutes}, $\eps_{i+\ell}(N\triangledown L(i^m))=\eps_{i+\ell}(N)=\eps$. Applying Lemma \cite[Lemma 3.5]{GV} again, we see that if $L$ is a composition factor of $(N\triangledown L(i^m))\circ L((i+\ell)^m)$ which is not equal to $(N\triangledown L(i^m))\triangledown L((i+\ell)^m)$, then
$\eps_{\wi}\leq \eps_{i+\ell}(L)<m+\eps$. This completes the proof of 2).
\end{proof}

\bigskip

\section{The crystal $\widehat{B}(\infty)$}

In this section we shall give the main results Theorem \ref{mainthm1a} and Theorem \ref{mainthm1b} of this paper. Throughout we assume that $\ell\geq 2$, $q:=\xi\in F$ is a primitive $2\ell$th root of unity in $F$, $I:=\Z/2\ell\Z$.\medskip

We fix an embedding $\theta: \Z/\ell\Z\hookrightarrow I$, $i+\ell\Z\mapsto i+2\ell\Z$ for $i=0,1,2,\cdots,\ell-1$. By some abuse of notations, for any $i\in \Z /\ell\Z$, we set $\wi:=\widehat{\theta(i)}=(\theta(i),\theta(i)+\ell+2\ell\Z)\in I\times I$. If $i_1,\dots,i_n\in\Z/\ell\Z$, then we define
$$
L(\wi_1,\cdots,\wi_n):=L\bigl(\theta(i_1),\theta(i_1)+\ell+2\ell\Z,\cdots,\theta(i_n),\theta(i_n)+\ell+2\ell\Z\bigr).
$$
As a special case of \eqref{simples}, we have $L(\wi_1,\cdots,\wi_n)=\wf_{\wi_n}\cdots\wf_{\wi_1}\mathbf{1}$.


\begin{dfn}\label{infty} We define $$
\widehat{B}(\infty):=\bigl\{L(\wi_1,\cdots,\wi_n)=\wf_{\wi_n}\cdots\wf_{\wi_1}\mathbf{1}\bigm|n\in\N, i_1,\dots,i_n\in\Z/\ell\Z\bigr\},$$
For each $n\in\N$, we define $\widehat{B}(n):=\bigl\{L(\wi_1,\cdots,\wi_n)\bigm|i_1,\cdots,i_n\in \Z/\ell\Z\bigr\}$. In particular, $\widehat{B}(\infty)=\sqcup_{n\in\N}\widehat{B}(n)$.
\end{dfn}

Recall the definition of the automorphism $\tau$ of $\HH_n$ in Definition \ref{sigmatau}.

\begin{prop}\label{prop:invariant}
    $\widehat{B}(n)$ consists of the iso-classes of those simple modules $L$ in $\Rep_I\HH_{2n}$ such that $L^\tau \simeq L$.
\end{prop}

\begin{proof}
    We prove this by induction on $n$. For $n=0$, it is trivial. Assume that $n > 0$ and the result holds for $n-1$. By (\ref{commu1}) and characters consideration we can deduce that $L(i,i+\ell)^{\tau} \simeq L(i,i+\ell)$. For $M \in \widehat{B}(n)$, there exists $N \in \widehat{B}_{n-1}$, s.t. $M = N \triangledown L(i,i+\ell)$ for some $i$ by definition. As $N\simeq N^\tau$ by induction hypothesis, $(N \circ L(i,i+\ell))^{\tau} \simeq N \circ L(i,i+\ell)$. This implies that $M^\tau \simeq M$.

    Conversely, if $M\in\Rep_I\HH_{2n}$ is s.t. $M^\tau\cong M$, then $(e_iM)^\tau \simeq e_{i+\ell}M$ and therefore $\eps_i(M) \neq 0 \Leftrightarrow \eps_{i+\ell}(M) \neq 0$. For such an $i$, $(e_ie_{i+\ell}M)^\tau \simeq e_{i+\ell}e_{i}M \neq 0$, hence $\we_{\wi}M \in \widehat{B}(n-1)$ by induction hypothesis, and $M \in \widehat{B}(n)$.
\end{proof}

The following result is a key step in the proof of our main results Theorem \ref{mainthm1a} and Theorem \ref{mainthm1b} of this paper.

\begin{thm}\label{crystal1} Let $M\in\widehat{B}(n)$. Then \begin{equation}\label{3epsEqual}\eps_{\wi}(M)=\eps_i(M)=\eps_{i+\ell}(M),\end{equation}
and \begin{equation}\label{ClosedKashiwaAction} \wf_{\wi}M\in\widehat{B}(n+1),\quad\, \we_{\wi}M\in\widehat{B}(n-1)\cup\{0\}.\end{equation}
Furthermore, if $n\geq 1$ then $\we_{\wi}M\in\widehat{B}(n-1)$ if and only if $\eps_{\wi}(M)>0$.
\end{thm}

\begin{proof}
    $\we_{\wi}M\in\widehat{B}(n-1)\cup\{0\}$ has been proved in the last proposition and $\wf_{\wi}M\in\widehat{B}(n+1)$ follows from definition. To show the equalities \eqref{3epsEqual}, as $\eps_{\wi}(M)=\min\{\eps_i(M),\eps_{i+\ell}(M)\}$, it suffices to show the second one. This follows from the fact that $(e_i^mM)^\tau \simeq e_{i+\ell}^m(M^\tau )$.
\end{proof}

\medskip


\begin{dfn} Let $M\in\Rep_I\HH_n$ be a simple module, $i\in I$ and $0\leq m\leq n$. We define $$\begin{aligned}
&\we_i^*M:=(\we_i(M^\sigma))^\sigma,\quad \wf_i^*(M):=(\wf_i(M^\sigma))^\sigma\cong \head(L(i)\circ M),\\
&\eps_i^*(M):=\eps_i(M^\sigma):=\max\{m\geq 0|(\we_i^*)^mM\neq 0\},\\
&\we_{\wi}^*M:=\we_i^*\we_{i+\ell}^*M,\quad  \wf_{\wi}^*M:=\wf_i^*\wf_{i+\ell}^*M,\\
&{\eps}_{\widehat{i}}^*(M)=\max\bigl\{m\geq 0\bigm|\bigl(\widetilde{e}_{\wi}^*\bigr)^mM\neq 0\bigr\}.
\end{aligned}
$$
\end{dfn}

Let $i\in I$ and $M\in\Rep_I\HH_n$ be a simple module. By the non-degenerate version of \cite[Lemma 10.1.3]{Klesh:book}, we have that $\eps_i^*(\wf_iM)\in\{\eps_i^*(M),\eps_i^*(M)+1\}$, and for any $i\neq j\in I$, $\eps_i^*(\wf_j(M))=\eps_i^*(M)$.

\begin{lem}\label{criterion} Let $i\in I$ and $M\in\Rep_I\HH_n$ be a simple module. The following statements are equivalent: \begin{enumerate}
\item $\eps_i(\wf_i^*M)=\eps_i(M)+1$;
\item $M\circ L(i)\cong L(i)\circ M$;
\item $M\circ L(i)\cong L(i)\circ M$ is simple;
\item $\wf_iM\cong\wf_i^*M$.
\end{enumerate}
\end{lem}

\begin{proof} We shall prove that $(4)\Rightarrow (3)\Rightarrow (2)\Rightarrow (1)\Rightarrow (4)$. If $\wf_iM\cong\wf_i^*M$, then $M\triangledown L(i)\cong L(i)\triangledown M$. Note that $L(i)$ is a real simple module. Applying Lemma \ref{real3properties}, we can deduce that $M\circ L(i)\cong L(i)\circ M$ is simple. In particular, $$
\eps_i(\wf_i^*M)=\eps_i(\head(L(i)\circ M))=\eps_i(\head(M\circ L(i)))=\eps_i(\wf_iM)=\eps_i(M)+1 .
$$

Conversely, assume that $\eps_i(\wf_i^*M)=\eps_i(\wf_i M)=\eps_i(M)+1$. Applying \cite[Lemma 10.1.6]{Klesh:book}, $\we_i\wf_i^*M\cong M$, which implies that $\wf_iM\cong\wf_i^*M$.
\end{proof}

\begin{lem}\label{2commutes} Let $i\in I$ and $M\in\widehat{B}(n)$. The following statements are equivalent: \begin{enumerate}
\item $M\circ L(i)\cong L(i)\circ M$;
\item $M\circ L(i+\ell)\cong L(i+\ell)\circ M$;
\item $\eps_i(\wf_i^*M)=\eps_i(M)+1$;
\item $\eps_{i+\ell}(\wf_{i+\ell}^*M)=\eps_{i+\ell}(M)+1$.
\end{enumerate}
In this case, we have $\wf_{\wi}^*M\cong\wf_{\wi}M\in\widehat{B}(n+1)$.
\end{lem}

\begin{proof} By Lemma \ref{criterion}, it suffices to prove $(1)\Rightarrow (2)$. Suppose that $L(i)\circ M\cong M\circ L(i)$. Let $c_{q^\ell}$ be the automorphism of $\HH_n$ which is defined on the generators by $$
c_{q^\ell}:\quad T_i\mapsto T_i,\quad X_j\mapsto q^{\ell}X_j,
$$
for $i=1,\cdots,n-1, j=1,\cdots,n$.
The automorphism $c_{q^\ell}$ induces a covariant automorphism of categories: $c_{q^\ell}: \Rep_I\HH_n\cong\Rep_I\HH_n$ such that $c_{q^\ell}(L(i_1,\cdots,i_n))\cong L(i_1+\ell,\cdots,i_n+\ell)$ for any $i_1,\cdots,i_n\in I$. Since $\ell>1$, it follows from Lemma \ref{eijcommutes} that
$c_{q^\ell}(L(\wi_1,\cdots,\wi_n))\cong L(\wi_1,\cdots,\wi_n)$ for any $L(\wi_1,\cdots,\wi_n)\in\widehat{B}(n)$. In particular, $c_{q^\ell}(M)\cong M$. Therefore, $$
L(i+\ell)\circ M\cong c_{q^\ell}(L(i))\circ M\cong c_{q^\ell}(L(i)\circ M)\cong c_{q^\ell}(M\circ L(i))=M\circ L(i+\ell) .
$$
In this case, by Lemma \ref{real3properties} $$\begin{aligned}
\wf_{\wi}^*M&=\wf_i^*\wf_{i+\ell}^*M\cong L(i)\circ L(i+\ell)\circ M\\
&\cong M\circ L(i+\ell)\circ L(i)\\
&=\wf_i\wf_{i+\ell}M\\
&\cong \wf_{\wi}M\in\widehat{B}(n+1).
\end{aligned}$$
\end{proof}

\begin{cor}\label{2commutesb} Let $i\in I$ and $M\in\widehat{B}(n)$. The following statements are equivalent: \begin{enumerate}
\item $\eps_i^*(\wf_iM)=\eps_i^*(M)+1$;
\item $M\circ L(i)\cong L(i)\circ M$ is simple;
\item $M\circ L(i+\ell)\cong L(i+\ell)\circ M$;
\item $\eps_{i+\ell}^*(\wf_{i+\ell}M)=\eps_{i+\ell}^*(M)+1$.
\end{enumerate}
\end{cor}

\begin{proof} This follows from Lemma \ref{criterion}, Lemma \ref{2commutes}.
\end{proof}

The proof of the following theorem is the same as that of Proposition~\ref{prop:invariant} and we leave it to the readers.

\begin{thm}\label{hatStarEPS} Let $n\in\N$. Then $$\widehat{B}(n)=\{\wf_{\wi_n}^*\cdots\wf_{\wi_1}^*\mathbf{1}|i_1,\cdots,i_n\in\Z/\ell\Z\}.$$ For any $i\in I$ and $M\in\widehat{B}(n)$, $\wf_{\wi}^*M\in\widehat{B}(n+1)$ and $\we_{\wi}^*M\in\widehat{B}(n-1)\cup\{0\}$. Furthermore, $\eps_{\wi}^*(M)=\eps_i^*(M)=\eps_{i+\ell}^*(M)$, and $\we_{\wi}^*M\in\widehat{B}(n-1)$ if and only if $\eps_{\wi}^*(M)>0$.
\end{thm}

Using Theorems \ref{crystal1}, \ref{hatStarEPS} and the non-degenerated version of \cite[10.1.1--10.1.7]{Klesh:book} we can immediately get the ``hat" version of all the results in \cite[\S10.1]{Klesh:book}.

\begin{lem}\label{10.1} Let $M\in\widehat{B}(n)$.

1) For any $i\in\Z/\ell\Z$, $\eps_{\wi}(\wf_{\wi}^*M)\in\{\eps_{\wi}(M), \eps_{\wi}(M)+1\}$.

2) For any $i,j\in\Z/\ell\Z$ with $i\neq j$, we have $\eps_{\wi}(\wf_{\wj}^*M)=\eps_{\wi}(M)$.

\end{lem}

\begin{lem}\label{10.2} Let $M\in\widehat{B}(n)$ and $i,j\in\Z/\ell\Z$. Assume that $\eps_{\wi}(\wf_{\wj}^*M)=\eps_{\wi}(M)$, $\eps:=\eps_{\wi}(M)$.
Then $\we_{\wi}^\eps \wf_{\wj}^*M\cong \wf_{\wj}^*\we_{\wi}^\eps M$.
\end{lem}

\begin{lem}\label{10.3} Let $M\in\widehat{B}(n)$.

1) For any $i\in\Z/\ell\Z$, $\eps_{\wi}^*(\wf_{\wi}M)\in\{\eps_{\wi}^*(M), \eps_{\wi}^*(M)+1\}$.

2) For any $i,j\in\Z/\ell\Z$ with $i\neq j$, we have $\eps_{\wi}^*(\wf_jM)=\eps_{\wi}^*(M)$.

\end{lem}

\begin{lem}\label{10.4} Let $M\in\widehat{B}(n)$ and $i,j\in\Z/\ell\Z$. Assume that $\eps_{\wi}^*(\wf_{\wj}M)=\eps_{\wi}^*(M)$, $a:=\eps_{\wi}^*(M)$.
Then $(\we_{\wi}^*)^a \wf_{\wj}M\cong \wf_{\wj}(\we_{\wi}^*)^a M$.
\end{lem}

\begin{cor}\label{10.5} Let $M\in\widehat{B}(n)$ and $i\in\Z/\ell\Z$. Let $M_1:=\we_{\wi}^{\eps_{\wi}(M)}M$, $M_2:=(\we^*_{\wi})^{\eps^*_{\wi}(M)}$. Then
$\eps_{\wi}^*(M)=\eps_{\wi}(M_1)$ if and only if $\eps_{\wi}(M)=\eps_{\wi}(M_2)$.
\end{cor}

\begin{lem}\label{10.6} Let $M\in\widehat{B}(n)$ and $i\in\Z/\ell\Z$ satisfying $\eps_{\wi}(\wf_{\wi}^*M)=\eps_{\wi}(M)+1$. Then
$\we_{\wi}\wf_{\wi}^*(M)\cong M$.
\end{lem}

\begin{lem}\label{10.7} Let $M\in\widehat{B}(n)$ and $i\in\Z/\ell\Z$ satisfying $\eps_{\wi}^*(\wf_{\wi}M)=\eps_{\wi}^*(M)+1$. Then
$\we_{\wi}^*\wf_{\wi}(M)\cong M$.
\end{lem}

Let $\widehat{\mathfrak{sl}}_{2\ell}$ be the affine Lie algebra of type $\widehat{A}_{2\ell-1}$. Let $\{{\alpha}_i|i\in I\}$ (resp., $\{{h}_i|i\in I\}$) be the set of simple roots (resp., coroots) of $\widehat{\mathfrak{sl}}_{2\ell}$.
We define $$
{B}(\infty):=\{L(i_1,\cdots,i_n)|n\in\N, i_1,\cdots,i_n\in I\}.$$  For any $M=L(i_1,\cdots,i_n)\in B(\infty)$, where $i_1,\cdots,i_n\in I$, we define the weight function ``${\wt}$" by ${\wt}(M):=-\gamma$ such that $$
{\gamma}=\sum_{i\in I}{\gamma}_i{\alpha}_i,\quad \text{where}\,\,{\gamma}_{i}=\#\bigl\{1\leq j\leq n\bigm|i_j=i\bigr\},\,\,\forall\,i\in I .
$$
For each $i\in I$, we define $\varphi_{i}(M):=\eps_{i}(M)+\<{h}_i,{\wt}(M)\>$. Let $U_{v}(\mathfrak{\widehat{sl}}_{2\ell})^{-}$ be the negative part of the quantized enveloping algebra $U_{v}(\mathfrak{\widehat{sl}}_{2\ell})$ of the affine Lie algebra $\widehat{\mathfrak{sl}}_{2\ell}$.

\begin{lem}\text{(\cite{G})} The  set ${B}(\infty)$, the functions $\eps_{i},\varphi_{i},{\wt}$, together with the operators $\we_{i},\wf_{i}$ form a crystal in the sense of Kashiwara \cite[\S7.2]{Ka1}. Moreover, it is isomorphic to Kashiwara's crystal associated to the crystal base of $U_{v}(\mathfrak{\widehat{sl}}_{2\ell})^{-}$.
\end{lem}

Let $\{\hat{\alpha}_i|i\in\Z/\ell\Z\}$ (resp., $\{\hat{h}_i|i\in\Z/\ell\Z\}$) be the set of simple roots (resp., coroots) of the affine Lie algebra $\widehat{\mathfrak{sl}}_{\ell}$. Let $U_{v}(\mathfrak{\widehat{sl}}_{\ell})^{-}$ be the negative part of the quantized enveloping algebra $U_{v}(\mathfrak{\widehat{sl}}_{\ell})$ of the affine Lie algebra $\widehat{\mathfrak{sl}}_{\ell}$. Let $\hat{P}, \hat{Q}$ be the weight lattice and root lattice of $\widehat{\mathfrak{sl}}_{\ell}$ respectively. Let $\{\hat{\Lam}_i|i\in\Z/\ell\Z\}$ be the set of fundamental dominant weights of $\widehat{\mathfrak{sl}}_{\ell}$.

\begin{dfn} For any $M=L(\wi_1,\cdots,\wi_n)\in\widehat{B}(n)$, where $i_1,\cdots,i_n\in\Z/\ell\Z$, we define the weight function ``$\hat{\wt}$" by $\hat{\wt}(M):=-\gamma$ such that $$
\gamma=\sum_{i\in\Z/\ell\Z}\hat{\gamma}_i\hat{\alpha}_i,\quad \text{where}\,\,\hat{\gamma}_{i}=\#\bigl\{1\leq j\leq n\bigm|i_j= i\bigr\},\,\,\forall\,i\in
\Z/\ell\Z .
$$
For each $i\in\Z/\ell\Z$, we define $\hat{\varphi}_{\wi}(M):=\eps_{\wi}(M)+\<\hat{h}_i,\hat{\wt}(M)\>$.
\end{dfn}

\begin{thm}\label{mainthm1a} The set $\widehat{B}(\infty)$, the functions $\eps_{\wi},\varphi_{\wi},\hat{\wt}$, together with the operators $\we_{\wi},\wf_{\wi}$ form a crystal in the sense of Kashiwara \cite[\S7.2]{Ka1}.
\end{thm}

\begin{proof} This follows from Theorem \ref{crystal1} and Corollary \ref{KashiwaraOper} and the definitions of the functions $\eps_{\wi},\varphi_{\wi},\hat{\wt}$.
\end{proof}

Following \cite{Ka1}, for each $i\in\Z/\ell\Z$, we have the crystal $B_i=\{b_i(n)|n\in\Z\}$ with functions $$
\eps_j(b_i(n)):=\begin{cases} -n, &\text{if $j=i$,}\\ -\infty, &\text{if $j\neq i$;}
\end{cases}\quad
\varphi_j(b_i(n)):=\begin{cases} n, &\text{if $j=i$,}\\ -\infty, &\text{if $j\neq i$;}
\end{cases}.
$$
and $\wt(b_i(n)):=n\alpha_i$, and operators $$
\we_j(b_i(n)):=\begin{cases}b_i(n+1), &\text{if $j=i$,}\\ 0, &\text{otherwise;}
\end{cases}\quad
\wf_j(b_i(n)):=\begin{cases}b_i(n-1), &\text{if $j=i$,}\\ 0, &\text{otherwise;}
\end{cases}
$$

We set $b_i:=b_i(0)$. We define a map $\Psi_i: \widehat{B}(\infty)\rightarrow\widehat{B}(\infty)\otimes B_i$ which sends each $[M]\in\widehat{B}(\infty)$ to
$[(\we_{\wi}^*)^aM]\otimes\wf_i^ab_i$, where $a:=\eps_{\wi}^*(M)$.

The proof of the following two lemmas is completely the same as the proof of \cite[Lemma 10.3.1, Lemma 10.3.2]{Klesh:book} by using Lemmas \ref{10.1},
\ref{10.2}, \ref{10.3}, \ref{10.4}, \ref{10.5}, \ref{10.6}, \ref{10.7}. We leave the details to the readers.

\begin{lem}\label{klem1} Let $M\in\widehat{B}(\infty)$ and $i,j\in\Z/\ell\Z$ with $i\neq j$. We set $a:=\eps_{\wi}^*(M)$.\begin{enumerate}
\item[(i)] $\eps_{\wj}(M)=\eps_{\wj}\bigl((\we_{\wi}^*)^aM\bigr)$;
\item[(ii)] If $\eps_{\wj}(M)>0$, then $\eps_{\wi}^*(\we_{\wj}M)=\eps_{\wi}^*(M)$ and $(\we_{\wi}^*)^a\we_{\wj}M\cong\we_{\wj}(\we_{\wi}^*)^aM$.
\end{enumerate}
\end{lem}

\begin{lem}\label{klem2} Let $M\in\widehat{B}(\infty)$, $i\in\Z/\ell\Z$. Set $a:=\eps_{\wi}^*(M)$ and $L:=(\we_{\wi}^*)^aM$. \begin{enumerate}
\item[(i)] $\eps_{\wi}(M)=\max\{\eps_{\wi}(L),a-\<h_i,{\wt}(L)\>\}$;
\item[(ii)] If $\eps_{\wi}(M)>0$, then $$
\eps_{\wi}^*(\we_{\wi}M)=\begin{cases}a, &\text{if $\eps_{\wi}(L)\geq a-\<h_i,{\wt}(L)\>$;}\\
a-1, &\text{otherwise.}
\end{cases}
$$
\item[(iii)] If $\eps_{\wi}(M)>0$, then $$
(\we_{\wi}^*)^b\we_{\wi}M\cong\begin{cases}\we_{\wi}L, &\text{if $\eps_{\wi}(L)\geq a-\<h_i,{\wt}(L)\>$;}\\
L, &\text{otherwise,}
\end{cases} $$
where $b:=\eps_{\wi}^*(\we_{\wi}M)$.
\end{enumerate}
\end{lem}

\begin{thm}\label{mainthm1b} The crystal $\widehat{B}(\infty)$ is isomorphic to Kashiwara's crystal associated to the crystal base of $U_{v}(\mathfrak{\widehat{sl}}_{\ell})^{-}$.
\end{thm}

\begin{proof} This follows from a similar argument used in the proof of \cite[Theorem 10.3.4]{Klesh:book} via using Lemma \ref{klem1} and Lemma \ref{klem2} and \cite[Proposition 3.2.3]{KS}.
\end{proof}
\bigskip

\section{The crystal $\widehat{B}(\Lam_0+\Lam_\ell)$ and the Multiplicity Two Theorem}

Throughout we assume that $\cha F\neq 2$, $\ell\geq 2$, $q:=\xi\in F$ is a primitive $2\ell$th root of unity in $F$, $I:=\Z/2\ell\Z$.\medskip

Recall that $\widehat{\mathfrak{sl}}_{2\ell}$ is the affine Lie algebra of type $\widehat{A}_{2\ell-1}$, $\{{\alpha}_i|i\in I\}$ (resp., $\{{h}_i|i\in I\}$) is the set of simple roots (resp., coroots) of $\widehat{\mathfrak{sl}}_{2\ell}$, and ${P}, {Q}$ are the weight lattice and root lattice of $\widehat{\mathfrak{sl}}_{2\ell}$ respectively. Let ${Q}_n^+:=\{\alpha=\sum_{i\in I}\gamma_i{\alpha}_i\in{Q}|\sum_{i\in I}\gamma_i=n, \gamma_i\in\N, \forall\,i\in I\}$.
Let $\{{\Lam}_i|i\in I\}$ be the set of fundamental dominant weights of $\widehat{\mathfrak{sl}}_{2\ell}$.
Let ${P}^+:=\sum_{i\in I}\N{\alpha}_i$ be the set of integral dominant weights.

For any ${\Lam}\in {P}^+$, we define $J_{{\Lam}}$ to be the two-sided ideal of $\HH_n$ generated by $\prod_{i\in I}(X_1-\xi^i)^{\<{h}_i,{\Lam}\>}$. We define the non-degenerate cyclotomic Hecke algebra $\HH_n^{{\Lam}}$ to be the quotient $$
\HH_n^{{\Lam}}:=\HH_n/J_{{\Lam}} .
$$

In this section, we are mostly interested in the special level two case, i.e., when $\Lambda=\lam:={\Lam}_0+{\Lam}_{\ell}$.

\begin{dfn} The Iwahori--Hecke algebra $\HH_n^\lam$ of type $B_n$ is defined to be the quotient $\HH_n^\lam:=\HH_n/J_\lam$, where $J_\lam$ is the two-sided ideal of $\HH_n$ generated by $(X_1-1)(X_1+1)\in\HH_n$.
\end{dfn}

Following \cite{Klesh:book}, we have two natural functors $$
\pr^\lam: \Rep_I\HH_n\rightarrow\HH_n^\lam\lmod,\quad \infl^\lam: \HH_n^\lam\lmod\rightarrow\Rep_I\HH_n ,
$$
where $\pr^\lam M:=M/J_\lam M$, and $\infl^\lam$ is the natural inflation along the epimorphism $\pi_\lam:\HH_n\twoheadrightarrow\HH_n^\lam$. The functor $\infl^\lam$ is right adjoint to $\pr^\lam$.

Following \cite{G} and \cite{Klesh:book}, for each $i\in I$ and each simple module $M\in\HH_n^\lam\lmod$, we define the action of cyclotomic crystal operators: $$
\we_i^\lam M:=\pr^\lam\circ\we_i\circ\infl^\lam M,\quad
\wf_i^\lam M:=\pr^\lam\circ\wf_i\circ\infl^\lam M .
$$
By \cite{G}, we know that both $\we_i^\lam$ and $\wf_i^\lam$ define a map $B(\lam)\rightarrow B(\lam)\cup\{0\}$.

Recall that the set of $\Sym_n$-orbits $\bigl\{\Sym_n\cdot (0^{\gamma_0},\cdots,(2\ell-1)^{\gamma_{2\ell-1}})\bigm|\sum_{j=0}^{2\ell -1}\gamma_j=n\bigr\}$ is in one-to-one correspondence with the set $\{\sum_{i\in I}\gamma_i{\alpha}_i\in{Q}_n^+|\gamma_i\in\N,\forall\,i\in I\}$. Thus we can also use ${Q}_n^+$ to label the blocks of $\Rep_I\HH_n$. There is a decomposition $\HH_n^\lam\lmod\cong\oplus_{\gamma\in{Q}_n^+}(\HH_n^\lam\lmod)[\gamma]$, where $$
(\HH_n^\lam\lmod)[\gamma]:=\{M\in\HH_n^\lam\lmod|\infl^\lam M\in(\HH_n\lmod)[\gamma]\}.
$$
If $M\in(\HH_n^\lam\lmod)[\gamma]$ for some $\gamma=\sum_{i\in I}\gamma_i{\alpha}_i$, then we define $$
e_i^\lam M:=\begin{cases}(\res^{\HH_n^\lam}_{\HH_{n-1}^\lam}M)[\gamma-{\alpha}_i], &\text{if $\gamma_i>0$;}\\
0, &\text{if $\gamma_i=0$.}
\end{cases}
$$
while for any simple module $M\in\HH_n^\lam\lmod[\gamma]$ with $\gamma=\sum_{i\in I}\gamma_i{\alpha}_i$, we define $$
f_i^\lam M:=(\ind^{\HH_{n+1}^\lam}_{\HH_{n}^\lam}M)[\gamma+{\alpha}_i] .
$$
The functors $e_i^\lam, f_i^\lam$ are both left and right adjoint to each other, and hence are exact and send projectives to projectives.

\begin{dfn} Let $n\in\N$ and $M\in\Rep_I(\HH_{n})$. We define $$
\soc_{\widehat{B}}M:=\sum_{M\supseteq L\in\widehat{B}(\infty)}L .
$$
\end{dfn}
It is clear that $\soc_{\widehat{B}}M$ is a direct summand of $\soc M$.

\begin{thm}\text{(Multiplicity Two Theorem)}\label{mainthm2a} Let $i\in\Z/\ell\Z$ and $M\in\Rep_I\HH_n$ be a simple module, where $n\geq 1$. Then $e_{\wi}M$ is either $0$ or a self-dual indecomposable module with simple socle $\we_{\wi}M\cong \we_i \we_{i+\ell}M$. Furthermore, if $n\geq 2$, then $\soc_{\widehat{B}}\res_{\HH_{n-2}}^{\HH_n}M\neq 0$ only if $n$ is even and $M\in\widehat{B}(n/2)$. In that case, each simple module in $\soc_{\widehat{B}}\res_{\HH_{n-2}}^{\HH_n}M$ occurs with multiplicity two.
\end{thm}

\begin{proof} By \cite[Corollary 9.14]{G}, we can find a ${\Lam}\in{P}^+$ such that $\pr^{{\Lam}}M=M$. In other words, $M\in\HH_n^{{\Lam}}\lmod$. It follows that $e_iM\in\HH_{n-1}^\Lam\lmod, e_{i+\ell}M\in\HH_{n-1}^\Lam\lmod, e_ie_{i+\ell}M\in\HH_{n-2}^\Lam\lmod$. Therefore,
$e_{\wi}M\cong\infl^{\Lam}e_i^\Lam e_{i+\ell}^\Lam M$. Now the self-duality property of $e_{\wi}M$ follows from the non-degenerate version of \cite[Lemma 8.2.2]{Klesh:book}. Since $\ell>2$ and $\cha F\neq 2$, it follows from Proposition \ref{RedBao} that $e_{\wi}M$ is either $0$ or has simple socle $\we_{\wi}M\cong \we_i \we_{i+\ell}M$. Furthermore, if $\soc_{\widehat{B}}\res_{n-2}^{n}M\neq 0$ then clearly $n$ is even and $M\in\widehat{B}(n/2)$. In that case, since $\res_{\HH_{n-2}}^{\HH_n}M\cong\oplus_{i\in I}\res^{\HH_{n-1}}_{\HH_{n-2}}e_iM\cong\oplus_{i,j\in I}e_ie_jM$, it follows that each simple module in $\soc_{\widehat{B}}\res_{\HH_{n-2}}^{\HH_n}M$ occurs with multiplicity two (as both $\we_i\we_{i+\ell}M$ and $\we_{i+\ell}\we_iM$ occur).
\end{proof}

Let $\mathbf{1}_{\lam}\cong F$ be the trivial simple $\HH_0^\lam$-module. Let $i\in\Z/\ell\Z$ and $M\in\HH_n^\lam\lmod$ be a simple module. Recall from \cite{G} and \cite[\S8.4]{Klesh:book} that $$\begin{aligned}
\eps_i^\lam (M)&:=\max\{m\geq 0|(\we_i^\lam)^m M\neq 0\}=\eps_i(\infl^\lam M),\\
\varphi_i^\lam (M)&:=\max\{m\geq 0|(\wf_i^\lam)^m M\neq 0\} .
\end{aligned}
$$
In particular, $\eps_i(\mathbf{1}_\lam)=0$ and $\varphi_i(\mathbf{1}_\lam)=\<h_i,\lam\>$. If furthermore, $M\in\HH_n^\lam\lmod[\gamma]$ for some $\gamma=\sum_{i\in I}\gamma_i{\alpha}_i$, then we define the weight function $\wt^\lam(M):=\lam-\gamma$. Then $\varphi_i^\lam (M)=\eps_i^\lam (M)+\<h_i,\lam-\gamma\>$. We define $$
B(\lam):=\{M|\text{$M\in\HH_n^\lam\lmod$ is simple, $n\in\N$}\}. $$

\begin{lem}\text{(\cite{G})}\label{lamCrystal} The  set $B(\lam)$, the functions $\eps_{i}^\lam,\varphi_{i}^\lam,\wt^\lam$, together with the operators $\we^\lam_{i},\wf^\lam_{i}$ form a crystal in the sense of Kashiwara \cite[\S7.2]{Ka1}. Moreover, it is isomorphic to Kashiwara's crystal associated to the crystal base of the integral highest weight module $V(\lam)$ over $U_{v}(\mathfrak{\widehat{sl}}_{2\ell})$.
\end{lem}

\begin{dfn} For each $i\in\Z/\ell\Z$ and each simple module $M\in\HH_n^\lam\lmod$, we define $$\begin{aligned}
{e_{\wi}}^{\Lam_0} M&:=e_i^\lam e_{i+\ell}^\lam M,\quad f_{\wi}^{\Lam_0} M:=f_i^\lam f_{i+\ell}^\lam M .\\
{\we_{\wi}}^{\hat{\Lam}_0} M&:=\we_i^\lam\we_{i+\ell}^\lam M,\quad \wf_{\wi}^{\hat{\Lam}_0} M:=\wf_i^\lam\wf_{i+\ell}^\lam M .
\end{aligned}
$$
\end{dfn}

\begin{thm}\label{mainthm2b} Let $i\in\Z/\ell\Z$ and $M\in\HH_n^\lam\lmod$ be a simple module. Then

1) $\soc e_{\wi}^{\hat{\Lam}_0}M\cong\we_{\wi}^{\hat{\Lam}_0}M$. Furthermore, $e_{\wi}^{\hat{\Lam}_0}M\neq 0$ if and only if $\we_{\wi}^{\hat{\Lam}_0}M \neq 0$, in which case $e_{\wi}^{\hat{\Lam}_0}M$ is a self-dual indecomposable module with simple socle and head isomorphic to $\we_{\wi}^{\hat{\Lam}_0}M$;

2)  $\soc f_{\wi}^{\hat{\Lam}_0}M\cong\wf_{\wi}^{\hat{\Lam}_0}M$. Furthermore, $f_{\wi}^{\hat{\Lam}_0}M\neq 0$ if and only if $\wf_{\wi}^{\hat{\Lam}_0}M \neq 0$, in which case $f_{\wi}^{\hat{\Lam}_0}M$ is a self-dual indecomposable module with simple socle and head isomorphic to $\wf_{\wi}^{\hat{\Lam}_0}M$.
\end{thm}

\begin{proof} 1) follows from Theorem \ref{mainthm2a}. 2) follows from 1), Lemma \ref{lamCrystal} and the fact that $e_i^\lam, f_i^\lam$ are both left and right adjoint to each other.
\end{proof}

\begin{dfn}\label{Lam0} We define $\widehat{B}(\hat{\Lam}_0):=\{M=\wf_{\wi_n}^{\hat{\Lam}_0}\cdots\wf_{\wi_1}^{\hat{\Lam}_0}\mathbf{1}_\lam\neq 0|n\in\N, i_1,\cdots,i_n\in\Z/\ell\Z\}$. If $M=\wf_{\wi_n}^{\hat{\Lam}_0}\cdots\wf_{\wi_1}^{\hat{\Lam}_0}\mathbf{1}_{\lam}$ is non-zero, where $i_1,\cdots,i_n\in\Z/\ell\Z$, and $$\wt^\lam(M)=\lam-\sum_{i\in \Z/\ell\Z}\gamma_i({\alpha}_{\theta(i)}+{\alpha}_{\theta(i)+\ell}),$$ then we define  $$
\eps_{\wi}^{\hat{\Lam}_0}(M):=\eps_i^\lam(M)=\eps_{i+\ell}^\lam(M),\quad \wt^{\hat{\Lam}_0}(M):=\hat{\Lam}_0-\sum_{i\in\Z/\ell\Z}\gamma_i\hat{\alpha}_i,\quad
\varphi_{\wi}^{\hat{\Lam}_0}(M)=\eps_{\wi}^{\hat{\Lam}_0}(M)+\<\hat{h}_i,\hat{\Lam}_0-\sum_{i\in\Z/\ell\Z}\gamma_i\hat{\alpha}_i\> .
$$
\end{dfn}

Note that the above $\eps_{\wi}^{\hat{\Lam}_0}(M)$ is well-defined by Theorem \ref{crystal1}.

Let $M\in\HH_n^\lam\lmod$ be a simple module. We define $$
\soc_{\widehat{B}}M:=\sum_{M\supseteq L\in\widehat{B}(\lam)}L, \quad \head_{\widehat{B}}M:=\sum_{\head(M)\supseteq L\in\widehat{B}(\lam)}L .
$$

\begin{cor}\label{maincor2} Let $i\in\Z/\ell\Z$ and $M\in\HH_n^\lam\lmod$ be a simple module.

1) if $n\geq 2$ then $\soc_{\widehat{B}}\res_{\HH_{n-2}^{\lam}}^{\HH_n^{\lam}}M\neq 0$ only if $n$ is even and $M\in\widehat{B}(\hat{\Lam}_0)$. In that case, each simple module in $\soc_{\widehat{B}}\res_{\HH_{n-2}^\lam}^{\HH_n^\lam}M$ occurs with multiplicity two.

2)  $\head_{\widehat{B}}\ind_{\HH_{n}^{\lam}}^{\HH_{n+2}^{\lam}}M\neq 0$ only if $n$ is even and $M\in\widehat{B}(\hat{\Lam}_0)$. In that case, each simple module in $\head_{\widehat{B}}\ind_{\HH_{n}^\lam}^{\HH_{n+2}^\lam}M$ occurs with multiplicity two.
\end{cor}

\begin{proof} This follows from Theorem \ref{mainthm2b} and the non-degenerate version of \cite[Lemma 8.2.2]{Klesh:book}.
\end{proof}

\begin{thm}\label{mainthm3}\text{(\cite{G})} The set $\widehat{B}(\hat{\Lam}_0)$, the functions $\eps_{\wi}^{\hat{\Lam}_0},\varphi_{\wi}^{\hat{\Lam}_0},\wt^{\hat{\Lam}_0}$, together with the operators $\we^{\hat{\Lam}_0}_{i},\wf^{\hat{\Lam}_0}_{i}$ form a crystal in the sense of Kashiwara \cite[\S7.2]{Ka1}. Moreover, it is isomorphic to Kashiwara's crystal associated to the crystal base of the integral highest weight module $V(\hat{\Lam}_0)$ over $U_{v}(\mathfrak{\widehat{sl}}_{\ell})$.
\end{thm}

\begin{proof} Note that by definition it is easy to check that $\varphi_{\wi}^{\hat{\Lam}_0}(M)=\varphi_i^\lam(M)=\varphi_{i+\ell}^\lam(M)$ for any $M\in\widehat{B}(\hat{\Lam}_0)$. Now the theorem follows from a similar argument as that was used in the proof of \cite[Lemma 10.2.1, Theorem 10.3.4]{Klesh:book} by using Theorem \ref{crystal1}, Theorem \ref{mainthm1b} and \cite[Propositions 8.1, 8.2, Theorem 8.2]{Ka1}.\end{proof}

Finally, we remark that although we chose a level two cyclotomic quotient for $\lam=\Lam_0+\Lam_\ell$ and consider the Hecke algebras of types $D_n$ and $B_n$ in this paper, the construction and the main results of this paper should be able to be generalized to the cases of the cyclotomic Hecke algebras of type $G(pd,p,n)$ and some cyclotomic Hecke agebra of type $G(pd,1,n)$, and the dominant weight $\lam$ should be replaced with some ${\rm h}$-symmetric dominant weights in $P^+$ in a suitable sense, where ${\rm h}$ is the map defined in \cite[Theorem 4.2]{Hu3}.

\end{document}